\documentclass[10pt,oneside,english]{article}
\usepackage[english]{babel}
\usepackage[utf8]{inputenc}

\usepackage{graphicx} %% REMOVE [demo] to show figures!!!

\usepackage{geometry}
\usepackage{amssymb}
\usepackage{amsmath, amsthm}
\usepackage{amsmath,amscd}

\usepackage{hyperref}
\usepackage{color}
\usepackage{float}
\usepackage{subfig}

\usepackage{appendix}

\usepackage{hyperref}
\usepackage{multirow}

\makeatletter
 \theoremstyle{plain}
 \newtheorem{thm}{Theorem}[section]
 \numberwithin{equation}{section} %% Comment out for sequentially-numbered
 \newtheorem{prop}[thm]{Proposition} %%Delete [thm] to re-start numbering
 \theoremstyle{definition}
 \newtheorem{defn}[thm]{Definition}

 \theoremstyle{remark}
 \newtheorem{rem}[thm]{Remark}
\newtheorem{notation}[thm]{Notation}

\newcommand {\R}{\mathbb{R}}
\newcommand {\Z}{\mathbb{Z}}
\newcommand {\g}{\mathfrak{g}}

\newcommand {\lmet} {\langle \! \langle}
\newcommand {\rmet} {\rangle \! \rangle}

\newcommand{\J}{\mathbf{J}}
\newcommand{\tr}{\operatorname{tr}\,}
\newcommand{\diag}{\operatorname{diag}}

\newcommand{\restr}[1]{\vrule height3ex width.4pt depth1.4ex\lower1.4ex\hbox{\scriptsize $\,#1$}}
\newcommand{\rrestr}[1]{\vrule height2ex width.4pt depth0.9ex\lower0.9ex\hbox{\scriptsize $\,#1$}}
\newcommand{\II}{\mathbb{I}}
\newcommand{\ed}{\mathbf{d}}   % exterior diferential

\newcommand{\Proj}{\mathbb{P}} % Projector

\newcommand{\Ad}{\mathrm{Ad}}  % Adjoint action
\newcommand{\ad}{\mathrm{ad}}  % adjoint action
\newcommand{\D}{\mathbf{D}}    % capital derivative

\newcommand{\Vrig}{\mathcal{V}_\mathrm{RIG}}
\newcommand{\Vint}{\mathcal{V}_\mathrm{INT}}
\newcommand{\vecR}{\mathbf{R}}
\newcommand{\vecr}{\mathbf{r}}

\newcommand{\Ar}{\mathbf{Ar}} % Arnold form
\newcommand{\Sm}{\mathbf{Sm}} % Smale form

\newcommand{\Gaxi}{G^\mathrm{axi}} %axisymmetric group
\newcommand{\gaxi}{\mathfrak{g}^\mathrm{axi}} %axisymmetric lie algebra
\newcommand{\Gammaaxi}{\Gamma^\mathrm{axi}} %axisymmetric group

\newcommand{\Gasym}{G^\mathrm{asym}} %asymmetric group
\newcommand{\gasym}{\mathfrak{g}^\mathrm{asym}} %asymmetric lie algebra
\newcommand{\Gammaasym}{\Gamma^\mathrm{asym}} %asymmetric group

\newcommand{\e}{\mathbf{e}} %unitary vector

\title{A Hamiltonian Study Of The Stability and Bifurcations For The Satellite Problem}
\date{}
\author{Mu\~noz-Lecanda, M., Rodr\'iguez-Olmos M., Teixid\'o-Rom\'an M.}

\begin{document}

\maketitle

\begin{abstract}\let\thefootnote\relax\footnote{matmcml@ma4.upc.edu\\
miguel.rodriguez.olmos@ma4.upc.edu\\
miguel.teixido@ma4.upc.edu\\
Department of Applied Mathematics IV.
Technical University of Catalonia.
Barcelona,
Spain.}
We study the dynamics of a rigid body in a central gravitational field modeled as a Hamiltonian system with continuous rotational symmetries following the geometrical framework of Wang \emph{et al}.
Novelties of our work are the use the Reduced Energy Momentum for the stability analysis and the treatment of axisymmetric bodies. We explicitly show the existence of new relative equilibria and  study their stability and
bifurcation patterns.

\end{abstract}
\tableofcontents
\section{Introduction}

This article studies the dynamics of two massive bodies of
finite extent subject to their mutual gravitational
attraction. We assume that one of the bodies, the primary, has
a mass much greater than the other, the satellite, and that
the primary body is spherically symmetric. The problem is then equivalent
to the dynamics of a single rigid body in a central gravitational
field. This system has the mathematical structure of a Hamiltonian
system with symmetries. The symmetry assumption allows for considerable
simplifications and provides a valuable approximation
from which additional effects can be further studied using perturbation theory.

In the traditional approach to this problem (see \cite{meirovitch1970methods}),
the additional assumption that the motion of the center of mass of the satellite is unaffected by the finite extent of the body was frequently used. That is, the orbit described by the center of mass of the satellite is the same as if the satellite were replaced by a point mass. This further assumption is known in the literature as the \emph{restricted problem}. In our geometric approach we will not adopt this assumption, we will consider the full coupling of the system, or \emph{unrestricted problem}.

Since the global dynamics of this system may be very complicated
we will focus on the study of its relative equilibria, which are dynamical
evolution orbits which are contained in  group orbits. Relative
equilibria act as organizing centers for the dynamics of a
Hamiltonian system with symmetries and therefore the study of the
set of relative equilibria gives important information about the
qualitative behavior of the dynamical flow. There are several
classical studies on the coupling between orbital and attitude
motion for the satellite motion but in most of them the natural geometric and group-theoretic
properties of the problem are not fully exploited. The first geometric treatment of this
system is \cite{wang1990hamiltonian}, where Poisson reduction
and the Energy-Casimir method were the main tools employed in giving a full description of relative equilibria for large
orbits. In this article we will extend that work to the low-orbit regime, and
in particular we provide an analytic description of  the complete set of relative equilibria that
exist for small values of the orbital radius.

One novelty of our geometric treatment is that  we also study the axisymmetric
case, where in addition to the spatial rotational symmetry there is a $S^1$ body symmetry corresponding to rotations about a symmetry axis of the satellite. This also expands the work  of \cite{wang1990hamiltonian}, where
only the spatial rotational symmetry was considered. This axisymmetric case has been previously studied in \cite{beck-phd} using the Energy-Casimir method, whereas we treat this case via the Reduced Energy
Momentum method.

For both the asymmetric and axisymmetric cases we have found a complete description of all the existing relative equilibria families. In the axisymmetric case we have shown that a family of conical equilibria exists for arbitrary large orbits. This is a new family of relative
equilibria that does not seem to have been described in the literature before. In fact, in  \cite{o2004steady}, based on Routhian reduction and numerical
continuations, it is suggested that there were no conical
equilibria for large orbits but that  conclusion is due to the fact that their continuation path only detected the conical equilibria for very small orbits. Had they used a different continuation path they would have found the conical equilibria.

\paragraph{Organization of the paper.}

In Section \ref{sec:model} we describe the system modeling the motion of a satellite under the second order potential. In Section \ref{sec:asym} we analytically characterize existence conditions for the different relative equilibria of the asymmetric satellite.  In particular, in Proposition \ref{prop-conical} we  explicitly parametrize  conical equilibria, a family of  motions, existing for small orbital radii, in which the center of the orbit does not coincide with the center of the potential. To our knowledge, until now only through numerical descriptions (\cite{wang1990hamiltonian,o2004steady}) of this phenomenon existed. In Subsection \ref{asym-isotropy} we find that the different families can be classified in a very clean way using only symmetry considerations. In Subsections \ref{asym-orth-stab}, \ref{asym-par-stab}, \ref{asym-con-stab} we use the Reduced Energy Momentum method to study the stability of the different families. In Subsection \ref{asym-bifurcations} we describe the bifurcation schemes for the different relative equilibria. It appears that
 this is the first analytic treatment in the low orbit regime.

In Section \ref{axisymmetric} we study the axisymmetric satellite. Propositions \ref{prop:axi-orth} and  \ref{conical-charact} give  explicit existence conditions for all possible relative equilibria. In particular showing that conical equilibria exist for any value of the orbital radius as opposed to the situation in the asymmetric case (see Remark \ref{rem conic radius}).
In Propositions \ref{cylstab}, \ref{hypstab}, \ref{conicalstab} and \ref{isolatedstab} the Reduced Energy Momentum is applied to all the different families to provide stability and instability ranges. In Subsection \ref{sec:axi-bifurcations} we describe the different bifurcations between the different families of relative equilibria. The
behavior observed is very similar to what happens in the
restricted problem (see \cite{beck-phd}). The main difference is that
the restricted problem can not predict the non-orthogonality of the conical
solutions.

\paragraph{Acknowledgements.}
The authors would like to acknowledge the financial support of the Ministerio de Ciencia e Innovaci\'on (Spain), project
MTM2011-22585 and AGAUR, project 2009 SGR:1338. M. Teixid\'o-Roma\'an also thanks the support of a FI-Agaur PhD Fellowship. M. Rodr\'{\i}guez-Olmos thanks the support of the EU-ERG grant ``SILGA".

\section{Model of a Rigid Body in a Central Field}
\label{sec:model}

\subsection{Configuration Space}
We will consider a spherically symmetric primary body of mass ${m}_1$ and a satellite of mass ${m}_2$ and arbitrary mass distribution. We will assume that  ${m}_1$ is much larger than ${m}_2$, that is, the satellite does not affect the motion of the primary. Therefore the primary will be at rest in an inertial reference.

With these assumptions the configuration space of the satellite in the central field generated by the primary is $Q=SO(3)\times \R^3$. A point in this manifold will be denoted in space coordinates by $q=(B,{\mathbf{r}})$. $B$ is the orthogonal matrix that realizes the transformation mapping the satellite from a reference configuration $\mathcal R$ to its orientation in space and ${\mathbf{r}}$ is the vector between the center of the field and the center of mass of the satellite. In body coordinates this configuration point is represented by $(B,{\mathbf{R}})$, where ${\mathbf{R}}=B^T {\mathbf{r}}$.

We denote a tangent vector $(\dot B,\mathbf{\dot r})$ at $(B,\mathbf{r})$, in space coordinates, as $(\Omega,\mathbf{\dot r})\in\mathbb{R}^3\times\mathbb{R}^3$. Here we have used the left trivialization of $T\mathrm{SO}(3)$ given by $\dot B=B\widehat\Omega$, where $\widehat \Omega$ is the $3\times 3$ matrix such that
$$\widehat\Omega u=\Omega\times u$$
for any $u\in\mathbb{R}^3$. Analogously, vectors in body coordinates are represented as $(\Omega,\mathbf{\dot R})$.

\subsection{Hamiltonian Description}
\label{hamiltoniandescr}
In order to simplify the algebraic manipulation of expressions that will appear, we will choose, according to \cite{beck-phd,wang1991steady}, the following units of mass, length and time
\begin{equation}m= m_2,\quad \ell=\sqrt{\frac{\mbox{tr}({\mathbf I})}{ m_2}}, \quad t=\sqrt{\frac{l^3}{{Gm_1}}}\label{units}\end{equation} where $\mathbf I$ is the body-fixed inertia tensor and $G$ is the gravitational constant. Note that with respect to these units the satellite has unit mass and its inertia tensor satisfies $\mbox{tr}(\mathbf I)=1$. All the expressions appearing in this paper will be referred to these units.

The kinetic energy metric is given by the sum of the rotational and translational kinetic energies of the satellite. In space coordinates

\begin{equation*} K(\Omega,\mathbf{\dot r})=\frac{1}{2}\left({\Omega}\cdot {\mathbf I}\  {\Omega} +  \dot{{\mathbf{r}}}\cdot \dot{{\mathbf{r}}}\right),\end{equation*}
while in body coordinates
\begin{equation*}K(\Omega,\mathbf{\dot R})=\frac{1}{2} \lmet (\Omega,\mathbf{\dot R}) ,(\Omega,\mathbf{\dot R}) \rmet \end{equation*}
where $ \lmet \cdot ,\cdot \rmet$ is the Riemannian metric on $\mathrm{SO}(3)\times\mathbb{R}^3$ given by
\begin{equation}
 \lmet \cdot ,\cdot \rmet =\left[\begin{array}{cc}
        \mathbf I - \widehat {\vecR} \widehat {\vecR}  &\widehat{\vecR}\\
	-\widehat {\vecR} & \mbox{Id}_3
       \end{array}
 \right]
\label{metric}.
\end{equation}
Without loss of generality will choose a basis such that $\mathbf I$ takes a diagonal form.

\begin{rem}
\label{inertiatensor}
The inertia tensor has the form $\mathbf I=\diag(I_1,I_2,I_3)$ , and since $\tr \mathbf I=1$,  two parameters are enough to determine it. Also its values must satisfy the usual triangular inequalities ($I_i<I_j+I_k$) that in this case imply  $0<I_i<\frac{1}{2}$.
\end{rem}

The potential energy corresponding to the inverse square force field can be written in the integral form:
\begin{equation}V(B,\mathbf r)=-\int_{B\mathcal R} \frac{dm_2(\mathbf r') }{|\mathbf r+\mathbf r'|}=-\int_\mathcal R \frac{dm_2(\mathbf r') }{|\mathbf r+B\mathbf r'|}=-\int_\mathcal R \frac{dm_2(\mathbf r')}{|\mathbf R+\mathbf r'|}=V(\mathbf R) \label{potexacte2}\end{equation}
where $\mathbf{ {r}'}$ is the vector from the center of mass of the satellite to the point with mass density $d{m}_2(\mathbf{{r}'})$. Note that in body coordinates $(B,\mathbf R)$ the potential does not depend on $B$, which is a consequence of the rotational symmetry of the system.

Usually, the radial distance $|{\mathbf{R}}+{\mathbf{r}}'|$ is going to be much larger than the dimensions of the satellite. Therefore we can expand $|{\mathbf{R}}+{\mathbf{r}}'|^{-1}$ in power series of ${r}^{-1}=|\mathbf{ {R} }|^{-1}$. After a lengthy computation in Cartesian coordinates (see \cite{meirovitch1970methods} for more details), the classical second order approximation to the potential is
\begin{equation}V_2(\mathbf R)=-\frac{1}{R}-\frac{1}{2R^3}+\frac{3}{2R^5}\mathbf R \cdot \mathbf I \mathbf R.\label{2potential}\end{equation}
The second order system will be the Hamiltonian system on $T^*(\mathrm{SO}(3)\times\mathbb{R}^3)$ governed by the Hamiltonian
\begin{equation}H_2=K+V_2\label{hamiltonian2}\end{equation} which is an approximation of the exact Hamiltonian \begin{equation}H_{\mbox{exact}}=K+V\label{exacthamiltonian}.\end{equation} Here we use the notation $K$ to define the corresponding kinetic energy induced in the fibers of $T^*\mathrm{SO}(3)$ from the Riemannian metric $\lmet\cdot,\cdot\rmet$ as usual. All throughout this article we will consider the second order model, with Hamiltonian $H_2$.

\subsection{Symmetries and Relative Equilibria} \label{symmetries-asym}

On the configuration space $Q=\mathrm{SO}(3)\times\mathbb{R}^3$ there is a natural action of $\Gasym_0=SO(3)$, given in space coordinates by $M\cdot (B,\vecr)=(MB,M\vecr)$ or, in body coordinates, by $M\cdot(B,\vecR)=(MB,\vecR)$, with $M\in\mathrm{SO}(3)$.

The fundamental fields of this action are given, for any $\xi\in\mathfrak{so}(3)\simeq\mathbb{R}^3$, by $$\xi_Q(B,\vecr)=\left.\frac{d}{dt}\right|_{t=0}(\exp(t\hat\xi)B,\exp(t\hat\xi)\vecr)=(\hat \xi B,\hat \xi \vecr),$$ which is represented in space coordinates by
$$\xi_Q(B,\vecr)=(B^T\xi,\xi\times\mathbf{r}).$$
and in body coordinates by
$$\xi_Q(B,\vecR)=(B^T\xi,0).$$

This action leaves the metric \eqref{metric} and the potentials \eqref{potexacte2}, \eqref{2potential} invariant, so \eqref{hamiltonian2} as well as the exact system \eqref{exacthamiltonian}, are $\Gasym_0$-symmetric simple mechanical systems (see the Appendix). The momentum map of the $\Gasym_0$-action on $T^*\mathrm{SO}(3)$ corresponds to the total angular momentum, the sum of the angular momenta due to orbital  and spinning motions.

From the above computations  the locked inertia tensor is given in body coordinates by
\begin{equation}\langle\xi, \mathbb I(B,\vecR)\eta\rangle=\lmet\xi_Q(B,\vecR), \eta_Q(B,\vecR) \rmet=\xi \cdot B(\mathbf I-\hat \vecR\hat \vecR)B^T\eta \implies \mathbb I(B,\vecR)= B(\mathbf I-\hat \vecR \hat \vecR)B^T\label{lockedI}\end{equation}

\begin{rem}It is interesting to describe what are the one-parameter subgroup orbits of $\Gasym_0$ in spatial coordinates. The orbit of the point $(B_0,\vecr_0)$ generated by the Lie algebra element $\xi$ is given by $B(t)=\exp(t\hat \xi)\cdot B_0$ and $\vecr(t)=\exp(t\hat \xi)\cdot \vecr_0$. But also:
$$\left|\vecr(t)-\frac{(\vecr(t)\cdot \xi)\xi}{|\xi|^2}\right|^2=|\vecr(t)|^2-\frac{\vecr(t)\cdot \xi}{|\xi|^2}=|\vecr_0|^2-\frac{\vecr_0\cdot \exp(-t\hat \xi)\xi}{|\xi|^2}=|\vecr_0|^2-\frac{\vecr_0\cdot\xi}{|\xi|^2}.$$
That is, the vector $\vecr$ describes a right circular cone with vertex at the origin and center at $C=\frac{\vecr_0\cdot \xi}{|\xi|^2}\xi$. Note that if $\xi\cdot \vecr_0\neq 0$ the orbit center $C$ does not coincide with the center of the potential $O$.
\end{rem}

\begin{defn}
A relative equilibrium such that $\vecr\cdot \xi=0$ is called \emph{orthogonal equilibrium}, otherwise it will be called \emph{non-orthogonal}.\label{defn-orthoblreleq}
\end{defn}

For orthogonal equilibria the center of mass the rigid body describes a circle whose center coincides with the center of the potential, (i.e. $C=O$) but in the non-orthogonal equilibria $C\neq O$ the orbit is a circle in plane not containing the center of the potential as explained in the previous remark.

\begin{rem}
There exist several different names for the orthogonal/non-orthogonal dichotomy in the literature. In \cite{o2004steady} it is distinguished between ``coplanar'' and ``non-coplanar'' motions, and in \cite{wang1990hamiltonian} the terminology ``great-circle motions'' and ``non-great circle motions''  is used.
\end{rem}
In \cite{stepanov1969steady}, the offset angle $\varkappa$ is defined as the angle of  $\vecr$ with the orbital plane, that is, \begin{equation}\label{defoffset}\sin \varkappa=\frac{\vecr_0\cdot \xi}{|\xi||\vecr_0|}.\end{equation} Although one may think that the only physically meaningful value for $\varkappa$ is $\varkappa=0$ this is not always the case. In fact the coupling between spinning and orbital motion can actually produce dynamical orbits with $\varkappa\neq 0$.

\subsection{Discrete Symmetries}
\label{sec:asym-discrete}

The second order model \eqref{hamiltonian2} has an additional set of discrete symmetries. Let $\Gammaasym$ be the group of transformations of $Q$ generated by the symmetry $$s:(B,\vecr)\mapsto (B,-\vecr),$$ and the 3 perpendicular rotations $$\rho_i:(B,\vecr)\mapsto (B\rho_i^T,\vecr),\quad i=1,2,3$$
where $\rho_i$ is a rotation of angle $\pi$ around the $i$-th principal axis of the satellite.

Using the following $3\times 3$ matrix representation of the generators of $\Gammaasym$  in body coordinates as
$$s=\left[\begin{array}{ccc}
        -1 & 0 & 0\\
	0 & -1 & 0 \\
	0 & 0 & -1
       \end{array}
 \right],\quad \rho_1= \left[\begin{array}{ccc}
        1 & 0 & 0\\
	0 & -1 & 0 \\
	0 & 0 & -1
       \end{array} \right],\quad \rho_2 =\left[\begin{array}{ccc}
        -1 & 0 & 0\\
	0 & 1 & 0 \\
	0 & 0 & -1
       \end{array}\right],\quad \rho_3 =\left[\begin{array}{ccc}
        -1 & 0 & 0\\
	0 & -1 & 0 \\
	0 & 0 & 1
       \end{array} \right],$$
 the action of $\Gammaasym$ on the configuration space $\mathrm{SO}(3)\times\mathbb{R}^3$ is $$A\cdot(B,\vecR)=(BA^T\det(A),A\vecR)$$

The total symmetry group of the second order model will be the direct product $\Gasym:=SO(3)\times \Gammaasym$ that acts on $\mathrm{SO}(3)\times\mathbb{R}^3$ as
\[(M,A)\cdot (B,\vecR) = (MBA^T\det(A),A\vecR)\]
with $(M,A)\in SO(3)\times \Gammaasym$.
The lift of this action to $T(\mathrm{SO}(3)\times\mathbb{R}^3)$ is given by
$$(M,A)\cdot (B,\vecR,\Omega,\dot \vecR)=(MBA^T\det(A),A\vecR;A\det(A)\Omega,A\dot \vecR).$$
It is easily verified that $(Q,\lmet,\rmet,\Gaxi,V_2)$ is a $\Gasym$-symmetric simple mechanical system.

\begin{rem}
Although $\Gasym_0$ acts freely on $\mathrm{SO}(3)\times\mathbb{R}^3$, the action of $\Gasym$ on $Q$ is only locally free, therefore exhibiting discrete stabilizers for certain points. This fact will be exploited in the subsequent classifications of relative equilibria.
\end{rem}

\begin{rem}
For the exact model  \eqref{exacthamiltonian} on can easily check  that  $\Gammaasym$ fixes the metric \eqref{metric}. However, the potential \eqref{potexacte2} will be invariant only if the mass distribution $dm_2(\vecr')$ is invariant under the transformations $s,\rho_1,\rho_2,\rho_3$. Therefore, in general, $\Gasym$ is not a symmetry of the exact model.
\end{rem}

\section{Asymmetric Case}
\label{sec:asym}

In this section we will assume that the three principal moments of inertia $I_1,I_2,I_3$ are pairwise different. A body satisfying this condition will be called \emph{asymmetric}.  In  Section \ref{axisymmetric} we will consider \emph{axisymmetric} bodies, which are those having two equal moments of inertia. In both cases we will study all the possible relative equilibria, as well as their stability and bifurcation properties.

\subsection{Existence of relative equilibria}
\label{existencereleq}
Relative equilibria pairs $(q,\xi)\in Q\times \gasym$ are characterized as those  for which $q$ is a critical point of the augmented potential  $V_\xi\in C^\infty(Q)$ (Theorem \ref{REthm}). By $\Gasym$-invariance we can assume, without loss of generality, that any critical point $q_e$ is of the form $q_e=(\text{Id},\vecR)$. If we call  $R=|\vecR|$, the vanishing of the first variation of $V_\xi$ evaluated at $B=\text{Id}$ is equivalent to the equations
\begin{subequations}
\label{eqordre2releq}
\begin{align}
 \widehat{\xi} (\mathbf I -\vecR\vecR ^T +R^2\text{Id}){\xi}&=0 \label{condicio1}\\
 \nabla_\vecR  V(\vecR ) +\xi(\vecR \cdot\xi)-\vecR |\xi|^2&=0 \label{condicio2}.
\end{align}
\end{subequations}
 We also have that in the second order approximation considered  \begin{equation}\nabla_\vecR V_2(\vecR )=\frac{\vecR }{R^3}+\frac{3\vecR }{2R^5}+\frac{3\mathbf I \vecR }{R^5}-\frac{15\vecR (\vecR \cdot \mathbf I \vecR )}{2 R^7}\label{gradient2n} .\end{equation} Therefore the relative equilibrium conditions for the second order model can be rewritten as
\begin{subequations}
\begin{align}
 (\mathbf I -\vecR \vecR ^T){\xi}=\alpha \xi\implies \vecR (\vecR \cdot\xi)=\mathbf I\xi-\alpha\xi\label{order2-1}\\
 \frac{\vecR }{R^3}+\frac{3\vecR }{2R^5}+\frac{3\mathbf I \vecR }{R^5}-\frac{15\vecR (\vecR \cdot \mathbf I \vecR )}{2 R^7} =-\xi(\vecR \cdot\xi)+\vecR |\xi|^2 \label{order2a}
\end{align}
\label{releq-2n}
\end{subequations}
for some real $\alpha$.

\subsection{Orthogonal Equilibria: Characterization}

If we assume $\xi \cdot \vecR =0$  equations \eqref{releq-2n} are greatly simplified. This condition defines orthogonal relative equilibria (see Definition \ref{defn-orthoblreleq}). The following result is reproduced  from  \cite{wang1990hamiltonian}, where it is proved that the orthogonality assumption is always satisfied if in the second order model we consider large orbital radii.

\begin{prop}
\label{wang-bound} (\cite{wang1990hamiltonian},\cite{beck-phd}) In the second order model \eqref{hamiltonian2}, for a fixed orbital radius $R>\frac{3}{2}$\begin{itemize}
 \item up to group translations there are exactly 6 different relative equilibria for any given $R$.
\item For each of them $\vecR$ and $B^T\xi$ are perpendicular and aligned with two different principal axis of the body.
\item The angular velocity is given by the modified Kepler's formula
\begin{equation}\label{modkepler}|\xi|^2=\frac{1}{R^3}+\frac{1}{2}\frac{3-9I_R}{R^5}\end{equation} where $I_R=(\vecR\cdot \mathbf I\vecR)/R^2$.
\end{itemize}
\end{prop}
The 6 possibilities arise from choosing a principal axis for $\vecR$ and then a perpendicular principal axis for $B^T\xi$. The orientation of $\vecR$ or $\xi$ along the principal axis is not relevant because it can be changed using a $\Gasym$-symmetry.

\begin{rem}
The first analysis of relative equilibria for this problem was given by Lagrange in \cite{lagrange}. In that reference, he did not analyze the second order system but a simplified one where he truncated the force and angular torque exerted on the satellite to the dominant term, finding that the translational and rotational motions decoupled in that approximation (in the literature this approximation is known as the \emph{restricted} problem \cite{beck-phd}). All the resulting relative equilibria happen to be orthogonal and the attitude is as in Proposition \ref{wang-bound} . That is, one of the principal axes is aligned with the radial direction and another is perpendicular to the orbital plane.
\end{rem}

\subsection{Parallel Equilibria: Characterization}
\label{par:charact}
If we  assume that $\mathbf R$ is an eigenvector of $\mathbf I$ equation \eqref{order2a} implies that either $\xi \cdot \vecR =0$ (orthogonal equilibria) or $\vecR\parallel\xi$. The first case is covered in the previous subsection In the second case
 \eqref{order2a} fixes $R$ but $\vert\xi\vert$ is arbitrary.
 This is equivalent to $$1+\frac{3}{2R^2}(1-3I_R)=0$$ where, as in  Proposition \ref{wang-bound}, $I_R=(\mathbf R \cdot \mathbf{IR})/R^2$. That is, there is a family of \emph{parallel equilibria} for each principal axis such that $I_i>\frac{1}{3}$.

\begin{rem}
\label{rem:parallel}
Although the case $\mathbf R\parallel\xi$ (parallel equilibrium) is a valid solution of equations \eqref{eqordre2releq} this is a shortfall of the second order model. Indeed, under this
 assumption \eqref{condicio2} gives $\nabla_{\mathbf{R}} V_2(\mathbf R)=0$. This means that the potential is neither attractive nor repulsive which is an impossibility for a  gravitational field.
This behavior is an artifact of the potential approximation \eqref{2potential}, for which the correction terms in $R^{-3}$ become dominant as $R$ approaches 0 making the approximation unphysical for very small $R$.
In fact from the above equation we can infer the upper bound
\begin{equation}\label{parallelbound}R^2<\frac34
\end{equation}
in order for parallel equilibria to exist.
For instance if we consider the International Space Station orbiting around the Earth using the second order model $\nabla_{\mathbf{R}} V_2(\mathbf R)=0$ will only happen when the distance the center of the ISS to the center of the Earth is about 50 m.

\end{rem}

\subsection{Conical Equilibria: Characterization}
For cases different from  $\xi \cdot \vecR =0$ or $\xi\parallel\vecR$ equations \eqref{releq-2n} can still be analytically solved.
\begin{prop}
There exist \emph{conical relative equilibria} for the second order system \eqref{hamiltonian2} for which $\mathrm{\vecR}$ and $\xi$ are neither parallel nor perpendicular.
In those solutions $\mathrm{\vecR}$ and $\xi$ belong to one of the principal planes determined by the inertia tensor $\mathrm{\mathbf{I}}$. If a basis of principal axes is chosen such
that $\mathbf R$ and $\xi$ are in the plane spanned by the first two
eigenvectors of $\mathbf I=\diag(I_1,I_2,I_3)$ then $\mathbf R=(R\cos \psi,R\sin
\psi, 0)$ satisfies \[A_4 S^2+A_2S+A_0=0 \text{ and } \nabla_{\mathbf R} V_2(\mathbf R)\cdot \mathbf R>0\] with $S=\cos^2\psi$ and
\begin{subequations}
\begin{align}
A_4&=225 (I_2-I_1)^2,\\
A_2&=6(I_2-I_1)(19R^2+15I_1-60I_2+15),\label{oblconditions}\\
A_0 &=(2R^2-9I_2+3)(8R^2+6I_1-15I_2+3).
\end{align}
\end{subequations}
 For this relative
equilibrium, if $\nabla_{\mathbf R} V_2(\mathbf R)=(g_1,g_2,0)$ then the associated
angular velocity is $\xi=k^{-1} (-g_2,g_1,0)$ where $k\in
\R$ satisfies
\begin{equation}k^2=\frac{|R|^2||\nabla_{\mathbf R} V_2(\mathbf R)|^2-((g_2,-g_1,0)\cdot \mathbf R)^2}{\nabla_{\mathbf R} V_2(\mathbf R)\cdot \mathbf R}.\label{k2}\end{equation}
\label{prop-conical}
\end{prop}

\begin{proof}
From  \eqref{order2-1} we have that if the equilibrium is not orthogonal, then $\mathrm{\vecR}$ must belong to the plane spanned by $\xi$ and $\mathbf I \xi$. Substituting this in \eqref{order2a} gives a linear relationship between $\xi,\,\mathbf I \xi$ and $\mathbf I^2 \xi$. This implies that if we take a basis aligned with the principal axes of $\mathbf I$ the following matrix must be singular
\begin{equation*} \begin{bmatrix}
    \xi_1 & I_1\xi_1 & I_1^2\xi_1 \\
\xi_2 & I_2\xi_2 & I_2^2\xi_2 \\
\xi_3 & I_3\xi_3 & I_3^2\xi_3 \\
   \end{bmatrix}.
\end{equation*}
We see that this happens only if two $I_i$ are equal (axisymmetric case) or if some component of $\xi$ vanishes.
Since in this section we are studying the asymmetric case, we will only consider the second option. Without loss of generality we can reorder the basis vectors so that $\xi_3=0$. Using \eqref{order2-1} this implies $R_3=0$. Let $\mathbf R=(R\cos \psi, R\sin \psi,0)$. Using this assumption and \eqref{order2a} we have
\begin{eqnarray*}\xi\cdot \nabla_{\mathbf R} V_2(\mathbf R) & = & 0\\
\e_3\cdot \nabla_{\mathbf R} V_2(\mathbf R) & = & 0. \end{eqnarray*}
which implies that there is some $k\in \R$ such that $\nabla_{\mathbf R} V_2(\mathbf R)=k(-\xi_2,\xi_1,0).$ Let $\nabla_{\mathbf R} V_2({\mathbf R})=(g_1,g_2,0)$. Note that $g_1,g_2$ are functions of $r$ and $\psi$ only. With these definitions \eqref{order2-1} is equivalent to
$$\left(\begin{bmatrix}
              I_1 & 0 \\ 0 & I_2
             \end{bmatrix}-\mathbf R \mathbf{R}^T\right)\begin{bmatrix}
                                  g_2 \\ -g_1
                                 \end{bmatrix}=\beta\begin{bmatrix}
                                  g_2 \\ -g_1
                                 \end{bmatrix}$$
for some $\beta\in \R$. After eliminating $\beta$  we obtain the equation
$$(A_4S^2 +A_2S +A_0)(I_1-I_2)\cos\psi \sin\psi=0$$
where $S=\cos^2\psi$ and $A_4,A_2,A_0$ are defined in \eqref{oblconditions}. If $\cos\psi\sin\psi=0$ then $\vecR$ is an eigenvector of $\mathbf{I}$ and as we saw above the solution is either a orthogonal or a parallel equilibrium. Therefore, we can assume that $A_4S^2 +A_2S +A_0=0$ and $\cos^2\psi$ is a function of $r$.

Now multiplying \eqref{order2a} with $\mathbf{R}$, after some simplifications we arrive to  \eqref{k2}. Note that this has a real solution if and only if $\nabla_\mathbf{R} V_2(\mathbf{R})\cdot \mathbf{R}>0$ or equivalently
\begin{equation}1+\frac{3}{2R^2}-\frac{9\mathbf{R}\cdot \mathbf {I R}}{2R^4}>0. \label{conicalbound}\end{equation}
\end{proof}
\begin{rem}
Notice that some authors (see for instance \cite{wang1990hamiltonian}) use the name \emph{oblique} equilibria correspon\-ding to our definition of conical equilibria.
\end{rem}
We have found algebraic conditions for the existence of conical equilibria. In order to obtain one such relative equilibrium one can proceed as follows.
Given a value of $R=|\mathbf R|$
\begin{itemize}
 \item Check if $A_4S^2+A_2S+A_0=0$ has a solution $0<S<1$. Let $S=\cos^2 \psi$
 \item From this $S$ construct $(g_1,g_2)$ using $(R,\psi)$ and check if $g_1 R\cos \psi +g_2R\sin \psi>0$. In that case obtain $k$ from \eqref{k2}.
 \item The relative equilibrium is characterized by $\mathbf R=(R\cos \psi, R\sin \psi,0)$ and $\xi=k^{-1} (-g_2,g_1,0)$.
\end{itemize}
\begin{rem}
\label{rem:beckbound}
 Actually in \cite{beck-phd} it is proved that  non-orthogonal equilibria could exist  only if
\begin{equation}1+\frac{3}{2R^2}-\frac{15\mathbf{R}\cdot \mathbf {I R}}{2R^4}<0. \label{beckbound}\end{equation}
As a consequence, the range of orbits in which conical equilibria exist is bounded by \eqref{beckbound} and \eqref{conicalbound}.
\end{rem}

\subsection{Orthogonal Equilibria. Stability}
\label{asym-orth-stab}
If $((\mathrm{Id},\vecR),\xi)\in Q\times \gasym$ is an orthogonal relative equilibrium, according to Proposition \ref{wang-bound} and using the $\Gasym$-symmetry, there is a basis such that the inertia tensor is $\mathbf{I}= \diag(I_1,I_2,I_3)$, the angular velocity is $\xi=(\xi_1,0,0)\in\mathfrak{so}(3)$, $\vecR=(0,R,0)$ where $\xi_1>0$, $R>0$ and the relationship between $R$ and ${\xi_1}$ is given by \eqref{modkepler}. We will now implement the Reduced Energy Momentum method (see the Appendix) in order to study its stability.

The metric \eqref{metric} and momentum at the equilibrium point are
\begin{equation*}\lmet\cdot ,\cdot \rmet_{(\mathrm{Id},\vecR)}=\begin{bmatrix}
               I_1+R^2 & 0 & 0 & 0 & 0 & R \\
0& I_2 & 0 & 0 & 0 & 0\\
0 & 0 & I_3+R^2 & -R & 0 & 0 \\
0& 0& -R & 1 & 0 & 0 \\
0 & 0 & 0 & 0 & 1 & 0\\
R & 0 & 0 & 0 & 0 & 1\\
              \end{bmatrix}, \quad \mu=\mathbb I(\text{Id},\mathbf R) \xi=\left[ \begin {array}{c}  \left( I_1+R^{2}
 \right) {\xi_1} \\0\\0
\end {array} \right].
\end{equation*}

\begin{notation}
From now on we will use the usual notation $\e_i$ for a vector with all the entries equal to zero except the $i$-th component being 1, being the total number of entries clear from the context.\end{notation}
With the above expressions we have
\begin{equation*}\gasym_\mu=\Big\langle\e_1\Big\rangle\implies (\gasym_\mu)^\perp=\Big\langle\e_2,\ \e_3\Big\rangle . \end{equation*}
Using this basis we  obtain the Arnold form \eqref{arnold}\begin{equation}\Ar(\eta,\nu)=\eta\cdot\left(-\hat\mu\mathbb I^{-1}(\text{Id},\mathbf R)\hat \mu+\hat\mu\hat\xi\right)\nu \implies
\Ar={\xi_1}^2\left[\begin{array}{cc}
   \frac{(I_1-I_3)(I_1+R^2)}{I_3+R^2} & 0 \\
   0 &  \frac{(I_1-I_2+R^2)(I_1+R^2)}{I_2}
  \end{array}\right].\label{ortharnold}
\end{equation}
Taking variations of the locked inertia tensor \eqref{lockedI} and evaluating at the relative equilibrium we have
$$\delta(\mathbb I\xi)(\text{Id},\mathbf R)=\begin{bmatrix}
0 & 0 & 0 & 0 & 2R {\xi_1} & 0 \\
0& 0 & {\xi_1}(I_1-I_2+R^2) & -R \xi_1 & 0 & 0\\
0 & {\xi_1}(I_3-I_1) & 0 & 0 & 0 &0
                       \end{bmatrix} \begin{bmatrix}\delta \boldsymbol\theta \\ \delta \mathbf{R} \end{bmatrix}$$
with $\delta\boldsymbol\theta, \delta \mathbf{R}\in \mathbb{R}^3$. Therefore one possible choice for the basis of $\Vint$ (see \eqref{Vint}) is

\begin{equation*}
\Vint = \Big\langle\e_5,\  -R_{{2}}\e_1+(R^2+I_1)\e_6,\  R\e_3+(R^2+I_1-I_2)\e_4  \Big\rangle \subset \mathcal{V}\subset T_{(\mathrm{Id},\vecR)}Q \cong \R^6.
\end{equation*}
In this basis the Smale form is diagonal and, after some simplifications, we arrive to
\begin{equation}
 \Sm=\xi_1^2\left[\begin{array}{ccc}
 \frac{3R^2-I_1}{R^2+I_1}-4+\frac{2}{R^3{\xi_1}^2} & 0 & 0 \\
  0 & \frac{3(I_1+R^2)^2(I_3-I_2)}{R^5{\xi_1}^2} & 0 \\
  0 & 0 & \frac{(I_1-I_2+R^2)(I_1-I_2)(6I_1+8R^2+3-15I_2)}{R^5{\xi_1}^2}\\
 \end{array}\right].\label{orthsmale}
\end{equation}

\begin{prop}
\label{lagrangestab}
There is a critical value $R_\mathrm{crit}>0$ such that all the orthogonal relative equilibria satisfying the \emph{Lagrange stability conditions}
\begin{equation}I_1>I_3>I_2\label{stabconditions}\end{equation} are $\Gasym_\mu$-stable if $R=|\mathbf R|>R_\mathrm{crit}$ and linearly unstable if $R<R_\mathrm{crit}$. In particular this critical value satisfies $R_\mathrm{crit}<2$.
\end{prop}

\begin{proof}
According to the Reduced Energy Momentum method (Proposition \ref{reducedEMsublk}) we only need to test the positiveness of the diagonal values of  \eqref{ortharnold} and \eqref{orthsmale}.
Up to positive factors $\Ar$ has eigenvalues $A_1=I_1-I_3, \quad A_2=1+(I_1-I_2)R^{-2}$ which, under the Lagrange conditions \eqref{stabconditions}, are both positive.
In a similar way, up to positive factors $\Sm$ has eigenvalues
\begin{align*}
S_1 &= \frac{3R^2-I_1}{R^2+I_1}-4+\frac{2}{R^3\xi_1^2}\\
S_2 &= I_3-I_2\\
S_3 &= (I_1-I_2)\left(1+\frac{I_1-I_2}{R^2}\right)(8R^2+6I_1+3-15I_2) .\\
\end{align*}
The Lagrange conditions \eqref{stabconditions} and $I_1+I_2+I_3=1$ (Remark \ref{inertiatensor}) imply that $I_2>\frac{1}{3}$ and $I_1<\frac{1}{3}$ and then $8R^2+3+6I_1-15I_2>8R^2>0$. That is, the Lagrange conditions imply that $S_2>0$ and $S_3>0$. Therefore
the only mechanism for loosing definiteness of $\Sm$ is when $S_1$ changes sign, but it follows from Proposition \ref{negativeeigenvalue} that changing a single eigenvalue implies instability. After some calculations, $S_1$ can be expressed as
\[S_1 = \frac{2R^4+(9I_2-6I_1-3)R^2-15I_1+45I_1I_2}{(I_1+R^2)2R^5\xi_1^2}.\]
Using the same ideas as with $S_3$ we find that $9I_2-6I_1-3<0$ and $-15I_1+45I_1I_2<0$. Now applying Descartes' rule of signs the equation $S_1=0$ has only one solution $R_\text{crit}$, that is, if $R>R_\text{crit}$ then $S_1>0$ and if $R<R_\text{crit}$ then $S_1<0$. Using again simple bounds for $I_1,I_2$ we find $S_1>0$ when $R=2$ so $R_\text{crit}<2$.
\end{proof}

\begin{rem}
Using the Reduced Energy Momentum method we have obtained precise bounds for the validity of the  Lagrange conditions.
In most physical applications the radius in the  chosen normalization (see \eqref{units}) is very large, and then in a first approximation (similar to the computations of \cite{wang1990hamiltonian}) we can neglect higher order terms of $R^{-1}$, therefore reobtaining the classical Lagrange stability conditions \cite{lagrange} for the \emph{restricted} problem.

We can relate the Lagrange conditions with the relative equilibria of the free rigid body. The stable relative equilibria for a free rigid body are the rotations around its axes of minimum or maximum inertia. However, in our problem, due to the existence of a central force field, only  rotations around the axis of maximum inertia are stable.
\end{rem}

\subsection{Parallel Equilibria. Stability}
\label{asym-par-stab}

The stability of the family of parallel equilibria can also be studied using the Reduced Energy Momentum method. Let $R=R^*$ be a solution for $1+\frac{3}{2R^2}(1-3I_2)=0$. Using the $\Gasym$-symmetry  there is a basis that diagonalizes $\mathbf{I}$ such that the equilibrium point can be expressed as
 $$\mathbf R=\begin{bmatrix} 0 \\ R^* \\ 0 \end{bmatrix},\quad \xi=\begin{bmatrix}0 \\ \xi_2 \\ 0\end{bmatrix},\quad \mu=\begin{bmatrix}0\\ I_2\xi_2 \\ 0\end{bmatrix}$$ where $\xi_2$ is a free parameter. A possible choice of bases adapted to the method is
\[\mathfrak{g}_\mu^\perp=\Big\langle\e_1,\ \e_3 \Big\rangle \subset \gasym\cong \R^3,\]
\[\Vint = \Big\langle -R^*\e_1+((R^*)^2+I_3-I_2)\e_6,\  R^*\e_3+((R^*)^2+I_1-I_2)\e_4,\ \e_5 \Big\rangle \subset \mathcal{V}\subset T_{(\mathrm{Id},\vecR)}Q \cong \R^6.\]
With respect to this basis $\Ar$ is diagonal with eigenvalues
\begin{equation}
A_1=-\xi_2^2\frac{I_2(7I_2+2I_3-3)}{9I_2+2I_3-3},\quad
A_2=-\xi_2^2\frac{I_2(7I_2+2I_1-3)}{2I_1+9I_2-3}
\label{pararnold}
\end{equation}
and $\Sm$ is also diagonal with eigenvalues
\begin{align}
S_1&=(I_2-I_3)\frac{(\xi_2^2(R^*)^5-3(R^*)^2+3I_2-3I_3)((R^*)^2-I_2+I_3)}{(R^*)^5}\nonumber\\
S_2&=(I_2-I_1)\frac{(\xi_2^2(R^*)^5-3(R^*)^2+3I_2-3I_1)((R^*)^2-I_2+I_1)}{(R^*)^5}\label{parsmale}\\
S_3&=\frac{2}{(R^*)^3}.\nonumber
\end{align}
Given these eigenvalues there are two values of $\xi_2$ that make the Hessian degenerate (4 if counted with signs). These correspond to the two values where bifurcations to the conical family occur as we will see in Section \ref{asym-bifurcations}.

\begin{rem}
The Arnold form in this case cannot
be positive definite, so \emph{parallel equilibria are not formally stable}. Indeed, suppose the Arnold form were positive definite then we would have
$$7I_2+2I_3-3<0\quad \text{ and }\quad 7I_2+2I_1-3<0. $$ Combining these two expressions we get $$14I_2+2I_3+2I_1<6\implies 12I_2<4 \implies I_2<\frac{1}{3}$$
but, as obtained in Subsection \ref{par:charact} $I_2$ would have to be greater than $\frac{1}{3}$ for the existence of parallel equilibria in the $\e_2$ direction, which is a  contradiction.

Certain values of the parameters $I_1,I_2,I_3$ can make the Reduced Energy Momentum Method inconclusive. Choosing several numerical values for those parameters we have computed the eigenvalues of the linearized Hamiltonian field and in all the cases we have found that at least one of them have a positive real part leading to instability. Unfortunately, we do not have an algebraic proof of this the case always.
\end{rem}

\subsection{Conical Equilibria. Stability}
\label{asym-con-stab}

The application of the Reduced Energy Momentum method
to the family of conical equilibria yields analytically intractable algebraic relations. However choosing several representative values of the parameters $I_1,I_2$, numerical estimations of the eigenvalues of $\ed^2 V_\mu$ and the linearization of the Hamiltonian vector field suggest that this family is unstable. We are not aware of rigorous results in the literature along these lines.

\subsection{Isotropy subgroups}
\label{asym-isotropy}

For each of the relative equilibria $(q,p)\in T^*(SO(3)\times \R^3)$ we can compute in a straightforward manner the isotropy groups of the base point $q$, the momentum value $\mu=\mathbf{J}(q,p)$ and the relative equilibria itself $z:=(q,p)$ with respect to the symmetry group $\Gasym$.  For example, using Proposition \ref{wang-bound} and the $\Gasym$-symmetry, an orthogonal equilibrium can be represented as the configuration point $q=(\mathrm{Id},(0,R,0))$ with angular velocity $\xi=(|\xi|,0,0)$. If $(M,A)\in G_q\subset \Gasym$ then
\[(\mathrm{Id},\vecR)=(M,A)\cdot (\mathrm{Id},\vecR)=(MA^T\det(A),A\vecR)\]
therefore, for this equilibrium, $A\in \Gammaasym$ must be of the form $\diag(\pm 1,1,\pm 1)$ and $M=A\det(A)$, that is, $\Gasym_q=\{(A\det(A),A) \mid  A=\diag(\pm 1,1,\pm 1)\}$ and as a group  is isomorphic to the Abelian group $\Z_2\times \Z_2$. In a similar way  we can obtain $G_z$ and $G_\mu$. This procedure can be applied to all the different relative equilibria, and  the resulting groups are shown in the following table.

\vspace{4mm}
\renewcommand{\arraystretch}{1.5}
\begin{tabular}{| c| l | l |}
\hline
\multirow{3}{2cm}{\centering $\text{Orth}^1_2$}   &  $\vecR=(0,R,0)$         &$\Gasym_q=\{(A\det(A),A) \mid  A=\diag(\pm 1,1,\pm 1)\}\cong \Z_2\times \Z_2$ \\
             &  $\xi=(|\xi|,0,0)\neq\mathbf{0}$       &$\Gasym_\mu=\{(\exp(t\e_1),A) \mid t\in \R\}\cong S^1\times (\Z_2)^3$ \\
             &  $\mu=(\mu_1,0,0)$                        &$\Gasym_z=\{(A\det(A),A) \mid A=\diag(\pm 1,1,1)\} \cong \Z_2$ \\
\hline
\multirow{3}{1cm}{\centering $\text{Par}_2$ $\xi\neq0$}   &  $\vecR=(0,R,0)$         &$\Gasym_q=\{(A\det(A),A) \mid  A=\diag(\pm 1,1,\pm 1)\}\cong \Z_2\times \Z_2$ \\
             &  $\xi=(0,|\xi|,0)$     &$\Gasym_\mu=\{(\exp(t\e_2),A) \mid t\in \R\}\cong S^1\times (\Z_2)^3$ \\
             &  $\mu=(\mu_1,0,0)$                      &$\Gasym_z=\{(\diag(a,1,a),\diag(a,1,a)) \mid a=\pm 1\} \cong \Z_2$ \\
\hline
\multirow{3}{1cm}{\centering $\text{Par}_2$ $\xi=0$}   &  $\vecR=(0,R,0)$         &$\Gasym_q=\{(A\det(A),A) \mid  A=\diag(\pm 1,1,\pm 1)\}\cong \Z_2\times \Z_2$ \\
             &  $\xi=(0,0,0)$        &$\Gasym_\mu=\Gasym\cong SO(3)\times (\Z_2)^3$ \\
             &  $\mu=(0,0,0)$                        &$\Gasym_z=\{(A\det(A),A) \mid  A=\diag(\pm 1,1,\pm 1)\}\cong \Z_2\times \Z_2$ \\
\hline
\multirow{3}{*}{$\text{Obl}_{1,2}$}   &  $\vecR=(R_1,R_2,0)$         &$\Gasym_q=\{(A\det(A),A) \mid  A=\diag(1,1,\pm 1)\}\cong \Z_2$ \\
             &  $\xi=(\xi_1,\xi_2,0)\neq\mathbf{0}$       &$\Gasym_\mu=\{(\exp(t\mu),A) \mid t\in \R\}\cong S^1\times (\Z_2)^3$ \\
             &  $\mu=(\mu_1,\mu_2,0)$                        &$\Gasym_z=\{(\mathrm{Id},\mathrm{Id})\}$ \\
\hline
\end{tabular}
\renewcommand{\arraystretch}{1}
\vspace{5mm}

The table above lists the different isotropy groups associated to one representative  for each family of relative equilibria.
Note that the isotropy groups can be used to classify the different families and gives information about the possible bifurcation schemes, since using standard topological arguments a family of relative equilibria can only bifurcate from another one whenever there is a subconjugation relationship between their isotropy groups $G_z$.
For example, we see that the family  $\text{Par}_2$ with $\xi\neq0$ cannot bifurcate from $\text{Orth}^1_2$ since points in the bifurcating branch must have stabilizers conjugated to a subgroups of the stabilizers of the original branch.\begin{rem}
The fact that the bilinear forms $\Ar$ and $\Sm$ in both the orthogonal \eqref{ortharnold},\eqref{orthsmale} and parallel \eqref{pararnold},\eqref{parsmale} equilibria are diagonal is not been a coincidence. This sub-blocking property is due to the fact that the bases chosen for $\gasym_\mu$ and $\Vrig$  correspond to what is called an isotypic decomposition for which an invariant bilinear form necessarily has to block-diagonalize (see Theorems 2.5 and 3.5 of \cite{golubitsky1988singularities2} for more details).
\end{rem}

\subsection{Bifurcations}
In this subsection we study the change in the stability regimes and bifurcation schemes for the different families of asymmetric relative equilibria.
\label{asym-bifurcations}
\begin{prop}\label{ortbifprop}With the notation of Proposition \ref{lagrangestab}, the  family of orthogonal relative equilibria satisfying Lagrange stability conditions loses stability at the point $R=R_\mathrm{crit}$ but no bifurcation occurs at that point.
\end{prop}

\begin{proof}
The degeneracy of the Arnold or Smale form at a relative equilibrium is a necessary condition for the existence of a bifurcation (see the Appendix) but this condition is not enough, and a further local study must be done in order to detect the existence of a bifurcating branch. As we showed in the proof of Proposition \ref{lagrangestab} the only eigenvalue of $\Ar$ and $\Sm$ that can vanish if the Lagrange conditions are satisfied is $S_1$, and this happens exactly when $R$ attains the value $R_\text{crit}$. We will now check, using the implicit function theorem, that at this point no bifurcating family exists. Equations \eqref{eqordre2releq} can be seen as a local map $F:\R^6\to \R^6$ that characterizes relative equilibrium pairs $(\vecR,\xi)$ as the set of points in $\R^6$ such that $F(\vecR,\xi)=0$. The family of orthogonal equilibria is a family of solutions of $F(\vecR,\xi)=0$ of the form
$$\xi=(\xi_1,0,0),\quad \mathbf{R}=(0,R_2,0)\quad\text{with}\,\,\xi_1^2=\frac{2R_2^2+3-9I_2}{2R_2^5}$$
parametrized by the value of $R_2$.
 We can compute the differential of $F$ evaluated at that point and consider the $5\times 5$ block corresponding to the variables $\xi_1,\xi_2,\xi_3,R_1,R_3$ leaving $R_2$ as a parameter.  After a lengthy computation the determinant of this $5\times 5$ block is
\begin{equation*}\Delta=K\xi_1^3(I_3-I_1)(I_2-I_3)(I_2-I_1)(8R_2^2+6I_1-15I_2+3) \end{equation*}
where $K\neq 0$ is a constant. If $\Delta\neq 0$ then the implicit function theorem ensures that all the solutions of $F(\vecR,\xi)=0$ near $R_2=RR_\mathrm{crit}$  can be parametrized by $R_2$. That is, near that point the only relative equilibria are orthogonal. Therefore, bifurcating families of relative equilibria from the orthogonal family can only appear when $\Delta$ vanishes. Finally, it is easy to see that  $\Delta$ does not have a zero at $R_2=R_\text{crit}$,
 and so no  bifurcation exists at that point.

\end{proof}

\begin{rem}
The fact that $S_1$ vanishes at $R_\text{crit}$ but no bifurcation occurs at that point can also be interpreted in the following way.
The square norm of the angular momentum as a function of $R$ along the orthogonal family is given by the expression
 $$|\mu|^2=|\mathbb I(\mathrm{Id},\vecR) \xi|^2=(I_1+R_2^2)^2\xi_1^2=(I_1+R_2^2)^2\frac{2R_2^2+3-9I_2}{2R_2^5}.$$
Also, after some simplifications $S_1$ can be written as
\begin{equation*}S_1=\frac{3R_2^2-I_1}{R_2^2+I_1}-4+\frac{2}{R_2^3\xi_1^2}= \frac{R_2}{|\mu|^2}\frac{d|\mu|^2}{dR_2}.\end{equation*}
Therefore, $S_1=0$ implies that the square norm of the momentum map attains a critical value at  $R_2=R_\text{crit}$. Denote by $\mu_\text{crit}$ this value. To understand what is happening at the point $R_\text{crit}$ we will need to make some observations:
The function $C:T^*Q\to \R$ defined by $C(z)=|\J(z)|^2$ induces a Casimir function $\mathcal{C}$ on the reduced Poisson manifold $T^*Q/SO(3)$ for which all the non-empty level sets $\mathcal{C}^{-1}(a)$  are symplectic leaves.
Moreover, the reduced Hamiltonian induced on $\mathcal{C}^{-1}(a)$ by \eqref{hamiltonian2}
has critical points when  $a >|\mu_\text{crit}|^2$
 but   no critical points if $a <|\mu_\text{crit}|^2$. Since critical points of the reduced Hamiltonian restricted to symplectic leaves are relative equilibria, we can say from a geometric point of view  that the orthogonal family, seen as a set of equilibria in the reduced space parametrized by the value of $\mathcal{C}([z])$,  undergoes a fold catastrophe as $\mathcal{C}(z)$ crosses $|\mu_\text{crit}|^2$.
This behavior
 also appears in the numeric studies of \cite{o2004steady} where they find that the type \texttt{Ia} motion (equivalent to our  Lagrangian orthogonal equilibrium) looses it stability as a critical distance is crossed.

\end{rem}

In fact, the analysis done in the proof of Proposition \ref{ortbifprop} proves that the orthogonal family can only bifurcate when $\Delta=0$. If the Lagrange conditions are satisfied, we showed  in Proposition \ref{lagrangestab} that $8R_2^2+6I_1-15I_2+3>0$, but on the other hand, as $I_2<\frac{1}{3}$, using \eqref{modkepler} we have that $\xi_1\neq0$ for any value of $R_2$. That is, $\Delta=0$ only if two moments of inertia are equal, this assumption being precisely the axisymmetric case that will be studied  in Section \ref{axisymmetric}.

We now study how the families of orthogonal  and parallel equilibria bifurcate.  Consider the family of orthogonal relative equilibria given by the representatives (see \eqref{modkepler}) $$\xi=(\xi_1,0,0),\quad \mathbf{R}=(0,R_2,0)\quad\text{with}\quad\xi_1^2=\frac{2R_2^2+3-9I_2}{2R_2^5}.$$

Note that if $3-9I_2>0$ this family contains elements for any value of $R_2>0$ but if $3-9I_2<0$ for $R_2$ small enough the conditions are empty, exactly at the point $\mathbf{R}=(0,R^*,0),\xi=(0,0,0)$ where \begin{equation}\label{R*}(R^*)^2=\frac{9I_2-3}{2}.\end{equation} As we did in the proof of Proposition \ref{ortbifprop} we will need to use an appropriate version of the implicit function theorem
 to understand what happens at that point. The relative equilibrium conditions \eqref{condicio1} and \eqref{condicio2} can be rewritten as
\begin{align*}
 (\mathbf I -\mathbf{R}\mathbf{R}^T +|\vecR|^2)\xi-\alpha\xi=0 \\
 \nabla_\mathbf{R} V(\mathbf{R}) +\xi(\mathbf{R}\cdot\xi)-\mathbf{R}|\xi|^2=0
\end{align*}
and this can be thought as the zero level set of a local map $F:\R^7\to \R^6$. This set of equations has as family of solutions given by $$\xi^*=(0,0,0),\quad \mathbf {R}^*=(0,R^*,0),\quad \alpha\in \R.$$

We will now use the implicit function theorem in order to study possible bifurcations. The matrix of partial derivatives with respect to $\xi_1,\xi_2,\xi_3,R_1,R_2,R_3$ is
$$
\begin{bmatrix}
 0 & 0 & 0 & I_1+(R^*)^2-\alpha & 0 & 0\\
 0 & 0 & 0 & 0 & I_2-\alpha & 0\\
 0 & 0 & 0 & 0 & 0 & I_3+(R^*)^2-\alpha\\
 * & 0 & 0 & 0 & 0 & 0 \\
 0 & * & 0 & 0 & 0 & 0 \\
0 & 0 & * & 0 & 0 & 0 \\
\end{bmatrix}
$$
where $*$ represents non-zero terms independent of $\alpha$. In view of this matrix the map $F$ can only bifurcate at the three points $\alpha=\alpha_1,\alpha_2,\alpha_3$ given by
\begin{equation*}\alpha_1=I_1+(R^*)^2, \quad \alpha_2=I_2,\quad \alpha_3=I_3+(R^*)^2 .\end{equation*}

The first bifurcation point corresponds $\alpha=\alpha_1$ to an orthogonal family spinning around $\e_1$ or $\e_3$ and with $\mathbf{R}$ aligned with $\e_2$.

 The third bifurcation point $\alpha=\alpha_3$ corresponds to an orthogonal family spinning around $\e_3$ and with $\vecR$ aligned with $\e_2$. Finally, the second  value  $\alpha=\alpha_2$ corresponds to the bifurcation
of  $\text{Par}_2$ ($\xi\neq 0$) from $\text{Par}_2$ ($\xi=0$) when $R\parallel\xi\parallel\e_2$.

\paragraph{Bifurcation diagram.} A simple sketch of what happens when $I_2>\frac{1}{3}$ is drawn in Figure \ref{simplebif}. If the orbital radius is large enough we have an orthogonal family of relative equilibria with $\xi$ aligned with $\mathbf {e}_1$ and $\mathbf R$ aligned with $\mathbf {e}_2$. This family is labeled as $\text{Orth}^1_2$. Also, if the orbital radius is large enough there is an additional orthogonal family  with $\xi$ aligned with  $\mathbf {e}_3$ and $\mathbf R$ aligned with $\mathbf {e}_2$. This family is labeled as $\text{Orth}^3_2$. As the radius $R$ gets smaller and reaches its critical value $R^*$ given by \eqref{R*}.
both orthogonal families $\text{Orth}^3_2$ and $\text{Orth}^1_2$ meet. Moreover this is a bifurcation point for the family of parallel equilibria spinning with $\mathbf R$ and $\xi$ aligned with $\e_2$ (labeled $\text{Par}_2$).

By a similar analysis (omitted here) we can obtain that the family $\text{Par}_2$ also intersects with
the family of conical equilibria in both the 12 and 23 planes (labeled $\text{Obl}_{1,2}$ and $\text{Obl}_{2,3}$). All those bifurcation points are represented in Figure \ref{simplebif}. Recall that each line represents a family of relative equilibria, or equivalently a 4 dimensional $SO(3)$-invariant submanifold of relative equilibria.

\begin{figure}[h]
\begin{center}
\includegraphics[width=10cm]{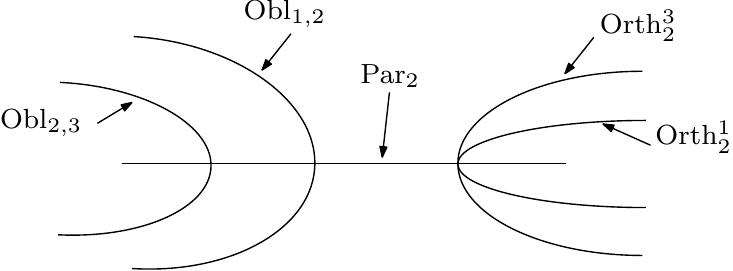}
\caption{Sketch of  the local bifurcations for the asymmetric body.\label{simplebif}}
\end{center}
\end{figure}

\section{Axisymmetric Case}\label{axisymmetric}

In this section we are interested in the case satellite has an axis of symmetry.
One of the seminal works studying this system is the article by Thomson \cite{thomson1962spin}, where the steady motion
of an axially symmetric satellite was investigated assuming that satellite's center of mass described a
prescribed circular Keplerian orbit, what we have called the restricted problem. In the relative equilibrium
found in that reference the axis of symmetry of the satellite was perpendicular to the orbital plane. In that configuration the satellite can spin about its
axis of symmetry with arbitrary velocity while still maintaining the circular orbit. The stability conditions for the  attitude motion are also investigated. Later, \cite{pringle1964bounds} and \cite{likins1966uniqueness}
showed that other steady motions were also present in the restricted problem. We will call these solutions the Pringle–Likins hyperbolic and conical  equilibria.

The unrestricted problem, where the attitude–orbit coupling is incorporated, was, to our knowledge, first examined in \cite{stepanov1969steady}. Among other results, the stability criteria for some of the unrestricted problem's counterparts to the Thomson and the Pringle–Likins hyperbolic equilibria are established. Also in that reference it is suggested that the unrestricted problem's counterparts to the Pringle–Likins conical equilibria could be examples of motions where the orbital plane of the satellite's center of mass does not contain the center of the potential. However the authors only provide an implicit description of those equilibria. Consequently, these and other questions concerning the existence and stability of steady motions for the unrestricted problem have so far remained open.

In \cite{beck-phd} it is conjectured the existence of an upper bound for $R$
in the conical family similar to the inequality stated in  Proposition \ref{wang-bound} for the asymmetric case, but this is not proved. In \cite{o2004steady} numerical continuation techniques are employed to follow the family of cylindrical equilibria  until it bifurcates to other equilibria. The only bifurcations found are families of hyperbolic and conical equilibria for very small radius, like the ones found in the asymmetric case. This led them to affirm that those motions can only exist for small orbital radius as it happens for  the conical equilibria for the asymmetric case.

In this section, we will show that the conical equilibria of Pringle–Likins have analogues in the unrestricted system and that they are not bound to small orbit. This contradicts the suggestion  of \cite{o2004steady} based  on numerical continuation experiments. Moreover, for a large region in the parameter space these motions are shown to be stable.

\subsection{$SO(3)\times S^1$ symmetry}
\label{axisym-action}
If the rigid body possesses an axis of symmetry then the symmetry group $G^\text{asym}$
introduced in Section \ref{symmetries-asym} can be augmented with the group of rotations around that axis, realized by a right action of $S^1$. If we choose a body fixed orthogonal frame such that the first element corresponds to the symmetry axis, then the inertia tensor will have the diagonal form $\mathbf I=\diag(I_1,I_2,I_2)$ and, as in Remark \ref{inertiatensor}, $I_1+2I_2=1$.

We can define an action of the direct product $SO(3)\times S^1=\Gaxi_0$ on the configuration space $Q$ by (in space coordinates) $$(M,R_\theta)\cdot(B,\vecr)=(MBR_\theta^T,M\vecr)$$ where $$R_{\theta}=\begin{bmatrix}1 & 0 & 0 \\ 0 & \cos \theta & -\sin \theta \\ 0 & \sin \theta & \cos \theta \end{bmatrix}=\exp(\theta \widehat{\e_1}).$$
The Lie algebra of $\Gaxi_0$ can be identified with the direct sum  $\R^3 \oplus \R$ in such a way that
the adjoint action is given by  $\ad_{(\xi,\eta)}(a,b)=(\widehat {\xi} a,0)$ for every $(\xi, \eta),(a,b)\in \gaxi$.

In body coordinates we have that the action is expressed as $$(M,R_\theta)(B,\mathbf{R})=(MBR_\theta^T,R_{\theta}\mathbf{R})$$
therefore the fundamental fields are
$$(\xi,\eta)_Q(B,\mathbf{R})=(B^T\xi-\eta \e_1,\eta \widehat{\e_1} \mathbf{R})=(\delta \theta,\delta \mathbf{R})\in T_{(B,\vecR)} Q\cong \R^6.$$
Finally, using this last expression and \eqref{metric} it is easy to check that the locked inertia tensor is given by
$$\mathbb I(B,\vecR)=\begin{bmatrix}
B(\mathbf I -\hat \vecR\hat \vecR)B^T & -B\e_1I_1\\
-(B\e_1I_1)^T & I_1 \\
  \end{bmatrix}$$

The relative equilibria of the system under this symmetry group are similar to the ones described in Section \ref{symmetries-asym}. The center of the satellite will move along a cone with vertex at the center of the potential while the satellite spins along its symmetry axis. If the relative equilibrium pair $((\mathrm{Id},\vecR),(\xi,\eta))\in Q\times \gaxi$ is \emph{orthogonal} in the sense that $\vecR \cdot \xi=0$ then the center of mass of the satellite will orbit along a circle with center the origin of the potential gravitational field.

\subsection{Discrete symmetries}
In Section \ref{sec:asym-discrete} we enlarged the symmetry group using the additional symmetries of the inertia ellipsoid. In the axisymmetric setting this will also happen. However this time the resulting group will not be a direct product.
Consider the matrix group
\[\Gammaaxi:=\left\{\begin{bmatrix}s & 0 \\ 0 & O\end{bmatrix}\mid  s\in \{1,-1\},\quad O\in O(2)\right\}\subset O(3)\]
and the map $\chi:\Gammaaxi \to \{1,-1\}$ defined by the determinant of the lower-right $2\times 2$ sub-matrix. Topologically $\Gammaaxi$ is the disjoint union of four copies of $S^1$ but as a group it is $O(2)\times \Z_2$.

We can define the following left action of $\Gaxi:=SO(3)\times \Gammaaxi$ on $Q$, in body coordinates, as
\[(M,N)\cdot (B,\vecR)=(MBN^T \det(N),N\vecR)\in SO(3)\times \R^3=Q.\]
The tangent lifted action is
\[(M,N)\cdot (\delta\theta,\delta\vecR)=(N\det(N)\delta\theta,N\delta\vecR)\in T_{(M,N)\cdot (B,\vecR)} Q\cong \R^6.\]
Using that $NR_\theta N^T=R_{\chi(N)\theta}$ the coadjoint action on the dual of its Lie algebra is
\[\Ad^*_{(M,N)^{-1}}(\mu,\nu)=(M\mu,\chi(N)\nu)\in {\gaxi}^*\cong \R^3\oplus \R.\]
From the above expressions one can check that the metric \eqref{metric} and the potential \eqref{2potential} are invariant under the action of the group $\Gaxi$. This implies that $(Q,\lmet,\rmet,\Gaxi,V_2)$ is a $\Gaxi$-symmetric simple mechanical system. Note that the connected component that contains the identity of $\Gaxi$ is $\Gaxi_0$ as defined before.
The action of $\Gaxi$ on $Q$ is locally free at $(B,\vecR)$ if $\vecR$ is not aligned with $\e_1$ but note that otherwise there exists continuous isotropy.

\subsection{Orthogonal Equilibria. Characterization}
As we in saw in Section \ref{existencereleq} the relative equilibria for the $\Gaxi$-action are the pairs $((B,\vecR),(\xi,\eta))\in Q\times \gaxi$ at which the first variation of the augmented potential $V_{(\xi,\eta)}$ vanishes. Again, without loss of generality we can assume that $B=\text{Id}$, and this gives the conditions
\begin{subequations}
\begin{align}
 \widehat{\xi} (\mathbf I -\widehat {\mathbf{R}} \widehat {\mathbf{R}}){\xi}-\widehat{\xi}I_1\e_1\eta=0 \label{axieta}\\
 \nabla_{\mathbf{R}} V_2(\mathbf{R}) +\xi(\mathbf{R}^T\xi)-\mathbf{R}\xi^T\xi=0 \label{axiV}
\end{align}
\label{axi-releq}
\end{subequations}
where $\nabla_{\mathbf{R}} V_2(\mathbf{R})$ has been computed in \eqref{gradient2n}.

As for the asymmetric case, if we assume $\xi \cdot \vecR =0$, then conditions \eqref{axi-releq} are greatly simplified, and we will refer to this case as orthogonal equilibria. The  case $\xi\cdot \vecR \neq 0$ will be addressed in Section \ref{sec:conical}.
Under the orthogonality assumption we can distinguish different families of relative equilibria.
\begin{prop}
Assume that a relative equilibrium $((\mathrm{Id},\vecR), (\xi,\eta))\in Q\times \gaxi$ for the axisymmetric second order model satisfies $\xi\cdot \vecR =0$. Then up to  translations by elements of $\Gaxi$ it belongs to one of the following cases

\begin{itemize}
 \item Cylindrical: There is $R>0$ and $\alpha\in \R$ such that
 \begin{equation}\xi=\begin{bmatrix}\xi_1 \\0 \\ 0 \end{bmatrix}\,quad \mathbf{R}=\begin{bmatrix}0 \\R \\ 0 \end{bmatrix},\quad |\xi|^2= \frac{2R^2+3-9I_2}{2R^5},\quad \eta=-\alpha\xi_1 \label{cylindricalpoint}. \end{equation}

 \item Hyperbolic: There is $R>0$ and $\theta\in S^1$, $\theta\neq \frac\pi2$ such that \begin{equation}\xi=|\xi|\begin{bmatrix}\sin\theta \\0 \\ \cos\theta \end{bmatrix},\quad \mathbf{R}=\begin{bmatrix}0 \\R \\ 0 \end{bmatrix},\quad |\xi|^2= \frac{2R^2+3-9I_2}{2R^5},\quad \eta=-\frac{I_2-I_1}{I_1}|\xi|\sin\theta.\label{hyperbolicpoint}\end{equation}

 \item Isolated: There is $R>0$ such that \begin{equation} \xi=\begin{bmatrix} 0 \\ \xi_2 \\ 0\end{bmatrix},\quad \mathbf{R}=\begin{bmatrix} R \\ 0 \\ 0\end{bmatrix}, \quad |\xi|^2= \frac{2R^2+3-9I_1}{2R^5},\quad \eta=0.\label{isolatedpoint}\end{equation}
\end{itemize}
\label{prop:axi-orth}
\end{prop}

\begin{proof}
Using $\xi\cdot \mathbf{R}=0$ in \eqref{axi-releq} we have
\begin{subequations}
\begin{align}
 \widehat{\xi} (\mathbf I {\xi}-I_1\e_1\eta)=0 \label{ortheta} \\
 \nabla_\mathbf{R} V_2(\mathbf{R}) -\mathbf{R}|\xi|^2=0. \label{orthV}
\end{align}
\label{axi-orth}
\end{subequations}

The first condition can be written as $\mathbf I {\xi}-I_1\e_1\eta=\lambda \xi $ which  in matrix form is $$\begin{bmatrix}
            I_1-\lambda & 0 & 0 \\ 0 &I_2-\lambda & 0 \\ 0& 0 & I_2-\lambda
           \end{bmatrix}\begin{bmatrix}\xi_1 \\ \xi_2 \\ \xi_3\end{bmatrix}=\begin{bmatrix}I_1\eta\\0 \\ 0 \end{bmatrix}.$$
There are several possibilities for the solutions of the above matrix equation.
\begin{itemize}

\item If $\lambda\neq I_1,I_2$ then the system has only one solution $\xi=((I_1-\lambda)^{-1} I_1\eta,0,0)$, and \eqref{orthV} forces $\mathbf{R}$ to be an eigenvector of the inertia matrix (see \eqref{gradient2n}). Using the orthogonality constraint we have $\mathbf{R}=(0,R\cos\theta,R\sin \theta)$ for some angle $\theta$. Using the $\Gaxi$-action we can assume $\mathbf{R}=(0,R,0)$. We can write this solution depending on the parameters $R>0$ and $\alpha\in \R$ as in \eqref{cylindricalpoint}. The parameter $\alpha$ will be called the \emph{spinning quotient}, and it is the quotient of the spinning velocity $\eta$ with the orbital angular velocity $\xi$.
 Note that the condition $\lambda \neq I_1,I_2$ is equivalent to $\alpha\neq \frac{I_2-I_1}{I_1},\ 0$.

\item If $\lambda=I_1$ the only solution is $\eta=0$, $\xi=(\xi_1,0,0)$. This is a solution without spinning. As in the cylindrical case using the available $\Gaxi$-symmetry and the orthogonality constraint we can assume that $\mathbf{R}=(0,R,0)$. This solution corresponds to \eqref{cylindricalpoint} with $\alpha=0$.

\item If $\lambda=I_2$ and $\eta\neq 0$ we obtain the family of solutions $\xi=(\frac{-I_1\eta}{I_2-I_1},\xi_2,\xi_3)$ where $\xi_2,\xi_3$ are arbitrary. As in the cylindrical case using the $\Gaxi$-symmetry and the orthogonality constraint we can assume that $\mathbf{R}=(0,R,0)$. Then  $\xi=|\xi|(\sin\theta,0,\cos\theta)$ and we get \eqref{hyperbolicpoint}.

\item If $\lambda=I_2$ and $\eta=0$ then $\xi=(0,\xi_2,\xi_3)$ for any $\xi_2,\xi_3$.and \eqref{orthV} implies that $\vecR$ is an eigenvector of $\mathbf{I}$ . There are two possibilities
\begin{itemize}
\item if $\vecR=(0,R_2,R_3)$ we can assume $\vecR=(0,R,0)$, but as $\xi\cdot \vecR=0$ and then $\xi=(0,0,\xi_3)$. This corresponds to a point in the hyperbolic family with either $\theta=0$ or $\theta=\pi$.

\item if $\vecR=(R,0,0)$, using the $S^1$ action we can assume that $\xi=(0,\xi,0)$ and the condition \eqref{orthV} gives the solution \eqref{isolatedpoint}.
\end{itemize}
\end{itemize}
\end{proof}
\begin{rem}
\label{rem:int-cyl-hyp}
Note that  the points along the cylindrical family with $\alpha=\frac{I_2-I_1}{I_1}$ are limiting cases (up to $\Gaxi$-translations) of the family of hyperbolic equilibria when $\theta\to\pm\frac{\pi}{2}$.
\end{rem}

\subsection{Non-orthogonal equilibria. Characterization}
\label{sec:conical}

If we assume that $\xi\cdot \mathbf{R}\neq0$  we will we obtain  new families of relative equilibria called parallel and conical equilibria for which in the latter the center of the orbit does not coincide with the center of the gravitational potential.
\begin{prop}
Let $((\mathrm{Id},\vecR), (\xi,\eta))$ be a relative equilibrium of the axisymmetric second order model, and assume that   $\vecR\cdot \xi \neq 0$. Then up to a $\Gaxi$-translations there are two possibilities\begin{itemize}
 \item  Conical: There exist $R>0$ and $\psi\in S^1$, $\psi\neq 0,\pm \frac{\pi}{2},\pi$ such that
 \[\vecR = R(\cos\psi,\sin\psi,0),\quad \xi=(\xi_1,\lambda R\sin\psi,0)\]
 where $\xi_1=3\frac{\cos\psi}{\lambda R^4}+\frac{I_1}{I_2-I_1}\eta$ and
 \begin{equation}
\eta=\frac{\cos\psi(I_2-I_1)}{2\sin^2\psi\lambda I_1R^6}\left((15\cos^2\psi-9)(I_1-I_2)-8R^2+6R^2\cos^2\psi+2\lambda^2R^7\sin^2\psi\right) \label{eq:eta}
\end{equation}

 \begin{equation}
 \lambda^2=\frac{\cos^2\psi \left(2R^2+(9-15\cos^2\psi)(I_1-I_2)\right)^2}
{2R^7\sin^2\psi\left(2R^2+(3-9\cos^2\psi)(I_1-I_2)\right)}. \label{eq:lambda}
\end{equation}

\item Parallel: $\vecR$ is an eigenvector of $\mathbf{I}$ satisfying $1+\frac{3}{2R^2}\left(1-3\left( \frac{\mathbf{R} \cdot \mathbf{IR}}{R^2}\right)\right)=0$, the spinning speed vanishes ($\eta=0$) and $\xi$ is an arbitrary multiple of $\vecR$. \end{itemize}
\label{conical-charact}
\end{prop}

\begin{proof}
From equations \eqref{axi-releq} and \eqref{gradient2n}, if we take the cross product of \eqref{axiV} with $\mathbf{R}$ and subtract it from \eqref{axieta}, we get the condition
$$\hat \xi\mathbf I \xi-\hat \xi I_1\e_1\eta-\hat{\mathbf{R}}\frac{3\mathbf {I R}}{R^5}=0$$
which in coordinates  is
$$\begin{bmatrix} 0 \\ \xi_3 I_1 \xi_1 -\xi_3 I_2\xi_1 -\eta \xi_3 I_1-3\frac{R_3}{R^5}(I_1-I_2)R_1\\
-\xi_2 I_1 \xi_1 +\xi_2 I_2\xi_1 +\eta \xi_2 I_1+3\frac{R_2}{R^5}(I_1-I_2)R_1\\
\end{bmatrix}=\begin{bmatrix}0 \\ 0 \\ 0\end{bmatrix}.$$
These two non-linear equations can be written as the system
$$\begin{bmatrix}\xi_3 & \frac{-3R_3}{|\vecR|^5}\\ -\xi_2 & \frac{3R_2}{|\vecR|^5} \end{bmatrix} \begin{bmatrix} \xi_1 \\ R_1\end{bmatrix} =\frac{I_1\eta}{I_2-I_1}\begin{bmatrix}-\xi_3\\ \xi_2\end{bmatrix}.$$
If the coefficients matrix is invertible then the solution is given by  $$\xi_1=\frac{I_1}{I_1-I_2}\eta,\quad R_1=0.$$ This relation and \eqref{axiV}
\begin{equation}\frac{\mathbf{R}}{R^3}+\frac{3\mathbf{R}}{2R^5}+\frac{3\mathbf{ I R}}{R^5}-\frac{15\mathbf{R}(\mathbf{R}\cdot \mathbf{ I R})}{2 R^7} +\xi(\mathbf{R}^T\xi)-\mathbf{R}\xi^T\xi=0\label{2ndeq}\end{equation}
forces either $\mathbf{R}\cdot \xi=0$ or $\mathbf{R}\parallel \xi$. The first case has already been covered in Proposition \ref{prop:axi-orth}. If $\mathbf{R}\parallel\xi$  then $\eta=0$ and we reobtain the condition for parallel equilibria in the asymmetric case (Subsection \ref{par:charact}) so we get the condition $1+\frac{3}{2R^2}\left(1-3\left( \frac{\mathbf{R} \cdot \mathbf{IR}}{R^2}\right)\right)=0$.

Assume now that the matrix of this system does not have full rank, that is, either $R_3=R_2=0$ or there is some $\lambda\in \R$ such that $\xi_2=\lambda R_2$ and $\xi_3=\lambda R_3$. If $R_3=R_2=0$, since $R=(*,0,0)$ then  \eqref{2ndeq} forces  either $\mathbf{R}\cdot \xi=0$ or $\mathbf{R}\parallel\xi$ as before, so there are no new cases. If $\xi_2=\lambda R_2$ and $\xi_3=\lambda R_3$ with $\lambda=0$ then the solutions of the  system are either $\xi=(*,0,0),\ \vecR=(0,*,*)$ or $\xi=(*,0,0),\ \vecR=(*,0,0)$ and again this does not offer new solutions.

Suppose now  that $\lambda \neq 0$. Using the $S^1$-action we can set $\mathbf{R}=(R_1,R_2,0)$. By the degeneracy assumption $\xi_2=\lambda R_2$ and $\xi_3=\lambda R_3$, $\xi_3=0$. The solution of the  system is now  $$\xi_1=3\frac{R_1}{\lambda R^5}+\frac{I_1}{I_1-I_2}\eta.$$ If we define the angle $\psi$ by $\mathbf{R}=(R\cos \psi,R\sin \psi,0)$, then the second component of the vector equation \eqref{2ndeq} is equivalent to
\begin{align*}-\lambda^2 (\sin^2 \psi \cos \psi) R^3 +\frac{I_1}{I_1-I_2}(\sin^2\psi)\lambda\eta R^2-\frac{1}{R^2}\left(3\cos^2\psi -4\right)\cos\psi \nonumber \\
-\frac{3}{2}\frac{\cos\psi}{R^4}\left((I_1-I_2)5\cos^2\psi-1-2I_1+5I_2\right)=0,\end{align*}
and $\eta$ can be solved. Substituting this $\eta$ in the third component of \eqref{2ndeq} gives $A+B\lambda^{-2}=0$ from which $\lambda$ can be easily found. The exact expressions for both variables are given in \eqref{eq:eta} and \eqref{eq:lambda}.
\end{proof}
For each radius $R$ and $\psi \in S^1$ ($\psi \neq 0,\pm \frac{\pi}{2},\pi$) there exists a conical equilibrium described  by Proposition \ref{conical-charact}. To understand the behavior of this equilibrium we can  expand the expressions for large $R$, obtaining
\begin{align*}
\mathbf{R}&= (R\cos\psi,R\sin\psi,0)\\
\xi &= (-R^{-\frac{3}{2}}\sin\psi,R^{\frac{3}{2}}\cos\psi,0) + O(R^{-\frac{7}{2}})\\
\eta &= 4\frac{I_2-I_1}{I_1}\sin\psi\, R^{-\frac{3}{2}}+O(R^{-\frac{7}{2}}).
\end{align*}
Note that with this approximation the orbit satisfies $\mathbf{R}\cdot \xi\approx 0$, but if higher order terms are taken into account then  \begin{equation}
 \sin^2\varkappa=\frac{(\xi\cdot \mathbf{R})^2}{\xi^2 R^2}=9(I_2-I_1)^2\sin^2\psi \cos^2\psi \frac{1}{R^4}+O\left(\frac{1}{R^6}\right)                                                                                                                                                                                \label{offset}.\end{equation}
where the offset angle $\varkappa$ is defined in \eqref{defoffset}.
\begin{rem}
Conical equilibria are orbits in a plane that does not contain the center of attraction. Although the offset is very small ($\varkappa$ decays like $R^{-2}$) this small value allows for the existence of this family of equilibria.
\end{rem}

\begin{rem}
In the case of three different moments of inertia the conical orbits can exist only for  very small radius where the second order potential looses its physical validity (see Remark \ref{rem:beckbound}). However, the situation is completely different for the axisymmetric case, in that the family of conical equilibria exists for arbitrary $R$.
\end{rem}

\begin{figure}[h]
\begin{center}
\subfloat{\includegraphics[width=2.54cm]{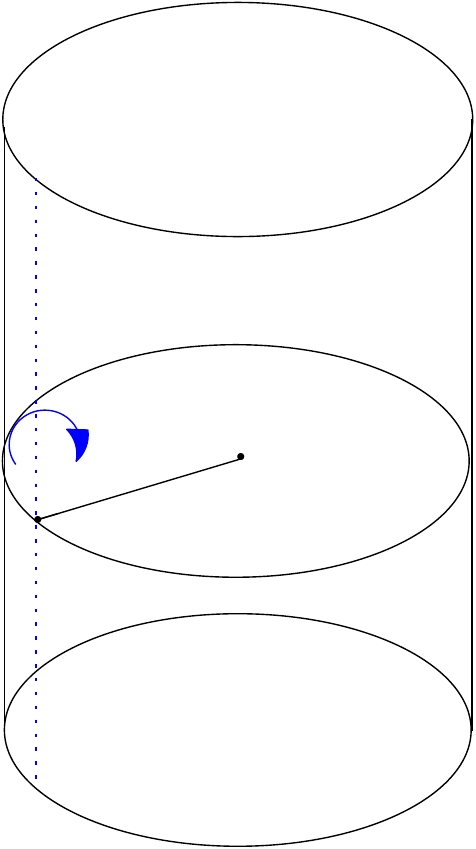}}\hspace{10mm}
\subfloat{\includegraphics[width=3.1cm]{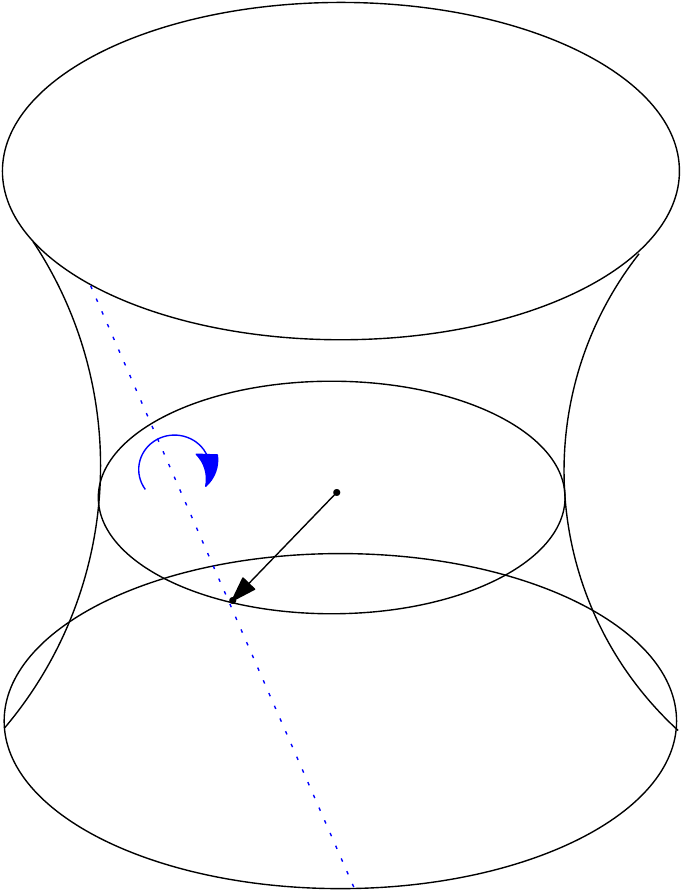}}\hspace{10mm}
\subfloat{\includegraphics[width=3.5cm]{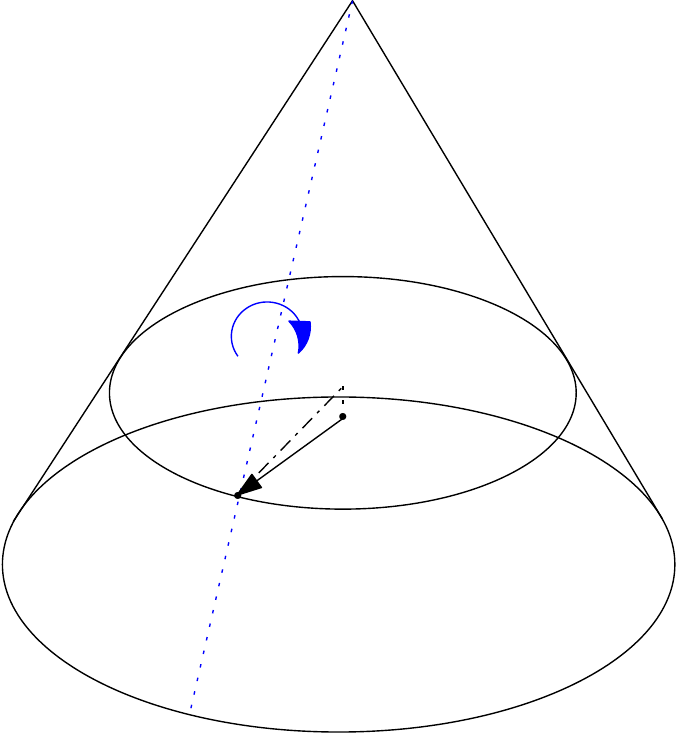}}

\caption{The cylindrical, hyperbolic and conical families of relative equilibria for the axisymmetric body. The evolution of the isolated equilibrium does not produce a three dimensional figure and has not been included.\label{releq_axi}}
\end{center}
\end{figure}

\begin{rem}\label{rem conic radius}
The names of these families of equilibria are based on the surface that describes the symmetry axis of the body as it travels along the orbit (see Figure \ref{releq_axi} where the dotted line represents the symmetry axis of the body).
In the classification introduced  in \cite{o2004steady}, cylindrical, hyperbolic and conical equilibria are called type I, III and IV respectively. The solutions of type II  are our parallel equilibria. It follows from the proof of Proposition \ref{conical-charact} that parallel equilibria only exist for small $R$ since the equations obtained for the derivation of this family are identical to those obtained for the parallel equilibria of the asymmetric case in Subsection \ref{par:charact}. In particular the bound \eqref{parallelbound} also holds for the parallel equilibria of the axisymmetric case.
\end{rem}

\begin{rem}
The cylindrical, hyperbolic and conical families all have  analogues in the restricted problem,  see \cite{pringle1964bounds}, \cite{likins1966uniqueness}.
\end{rem}

\subsection{Cylindrical equilibria. Stability}
We will now study the  family of cylindrical equilibria of the axisymmetric problem, where the satellite follows a circular orbit and in addition it spins with angular velocity parallel to the orbital angular velocity. Using  \eqref{cylindricalpoint} the metric and the angular momentum at the equilibrium point $(\mathrm{Id},\vecR)$ are
$$\lmet\cdot,\cdot\rmet_{(\mathrm{Id},\vecR)}=\begin{bmatrix} I_1+R^2 & 0 & 0 & 0 &0 &R\\
0 & I_2 & 0 & 0 & 0 & 0 \\
0 & 0 & I_2+R^2 & -R& 0 & 0 \\
0& 0 & -R & 1 & 0 &0 \\
0 &0 & 0 & 0 & 1& 0 \\
R_0 & 0 & 0 & 0 & 0 & 1\\
  \end{bmatrix}, \quad \mu=\II(\text{Id},\vecR)(\xi,\eta)=\begin{bmatrix}(I_1(1+\alpha)+R^2)\xi_1 \\0\\0\\-(1+\alpha)I_1\xi_1 \end{bmatrix},$$
therefore the stabilizer of the momentum value $\mu\in (\gaxi)^*$ is
\[\gaxi_\mu = \Big\langle\e_1,\ \e_4 \Big\rangle \implies (\gaxi_\mu)^\perp =\Big\langle\e_2,\ \e_3 \Big\rangle \subset \mathfrak{so}(3)\oplus \R \cong \R^4.\]
The Arnold form is
\begin{equation} \Ar= \begin{bmatrix}
\xi_1^2\frac{((1+\alpha)I_1+R^2)(I_1(1+\alpha)-I_2)}{I_2+R^2} &0 \\
0 & \xi_1^2\frac{((1+\alpha)I_1+R^2)(I_1(1+\alpha)+R^2-I_2)}{I_2+R^2}\\
\end{bmatrix}.\label{arnold-cyl}
\end{equation}
One possible choice for the Reduced Energy Momentum splitting is
\[\Vint=\Big\langle R\e_3+(R^2+I_1(1+\alpha)-I_2)\e_4,\ \e_5 \Big\rangle \subset \mathcal{V}\subset T_{(\mathrm{Id},\vecR)} Q\cong \R^6.\]
The Smale form in this basis is diagonal and has entries
\begin{equation*}
S_1=\frac{R^2+I_1(1+\alpha)-I_2}{R^5}\left(2R^2((4+\alpha)I_1-4I_2)+3(I_1-I_2)(I_1(1+\alpha)-I_2)\right),\quad S_2=\xi_1^2.
\end{equation*}
Using again Proposition \ref{reducedEMsublk} applied to the previous expressions we get

\begin{prop}
\label{cylstab}
For large orbits satisfying \begin{equation}\label{largeorbitcilindrical}R^2>-(1+\alpha)I_1,\end{equation} the cylindrical equilibrium with spinning quotient $\alpha\in \R$ satisfying $$I_1(1+\alpha)>I_2\quad \text{ and }\quad I_1(4+\alpha)-4 I_2>-\frac{3}{2R^2}(I_1-I_2)(I_1(1+\alpha)-I_2).$$
is $\Gaxi_\mu$-stable.

\end{prop}

\begin{rem}
Note that for any $I_1$ and $I_2$, if $\alpha$ is large enough then the cylindrical equilibrium will be $\Gaxi_\mu$-stable. This is the analogue of the fast top condition for the heavy top problem: the upright spinning equilibrium stable for high angular velocities.
\end{rem}

\begin{rem} Axisymmetric bodies can be classified as \emph{oblate} when $I_1>I_2$ or \emph{prolate} when $I_1<I_2$
 Note that  all cylindrical equilibria for oblate bodies satisfying $\alpha>0$ are $\Gaxi_\mu$-stable and in this case the large orbit condition \eqref{largeorbitcilindrical} is satisfied for any $R$.
\end{rem}

\subsection{Hyperbolic equilibria. Stability}
According to the characterization \eqref{hyperbolicpoint} of the family of hyperbolic equilibria, the metric and the angular momentum at the equilibrium point are
$$\lmet \cdot, \cdot \rmet_{(\mathrm{Id},\vecR)}=\begin{bmatrix} I_1+R^2 & 0 & 0 & 0 &0 &R\\
0 & I_2 & 0 & 0 & 0 & 0 \\
0 & 0 & I_2+R^2 & -R& 0 & 0 \\
0& 0 & -R & 1 & 0 &0 \\
0 &0 & 0 & 0 & 1& 0 \\
R & 0 & 0 & 0 & 0 & 1\\
  \end{bmatrix}\quad, \mu=\II(\text{Id},\vecR)(\xi,\eta)=\begin{bmatrix}(I_2+R^2)\xi\sin\theta \\ 0 \\ (I_2+R^2)\xi\cos\theta \\ -I_2\xi\sin\theta \end{bmatrix}.$$
Therefore we have
$$
\gaxi_\mu = \Big\langle\sin\theta\e_1+\cos\theta\e_3,\ \e_4 \Big\rangle$$
and $$(\gaxi_\mu)^\perp =\Big\langle\e_2,\  -(I_2+R^2)\cos\theta\e_1+R^2\sin\theta\e_3+(I_2+R^2)\sin\theta\e_4 \Big\rangle \subset \mathfrak{so}(3)\oplus \R \cong \R^4.
$$
With respect to this basis the Arnold form is diagonal and has eigenvalues
\begin{equation}A_1=(R^2+I_2)|\xi|^2\cos^2\theta I_2 R^{-2},\quad A_2=\frac{|\xi|^2R^2(R^2+I_2)}{I_2}\left(R^4+2R^2\cos^2\theta I_2+I_2^2\cos^4\theta\right) \label{arnold-hyp}.\end{equation}
We can choose the following basis of the subspace of internal variations \[\begin{split}
   \Vint=\Big\langle (R^2+I_2)(R\e_3+I_2\sin\theta\cos\theta\e_6)+ R^4+R^2I_2(1+\cos^2\theta) +I_2^2 \cos^2\theta)\e_4   , \hspace{1cm}\\
  ,-2R\sin\theta\e_2+(R^2+I_2)\cos\theta\e_5 \Big\rangle \subset \mathcal{V}\subset T_{(\mathrm{Id},\vecR)} Q\cong \R^6
  \end{split}
\]
and after some computations $\Sm$ takes diagonal form with eigenvalues
\begin{equation}S_1=3(I_1-I_2)(R^2+I_2\cos^2\theta)^2(R^2+I_2)^2R^{-5},\ S_2=2R^4-3R^2(I_1+I_2)+15I_2(I_2-I_1)\label{smale-hyp}.\end{equation}
We can now state the following stability result for hyperbolic equilibria
\begin{prop}\label{hypstab}  For large orbits satisfying
\begin{equation}R^2>\frac{1}{2}\left(3(1-I_2)+\sqrt{3+102I_2-351I_2^2}\right)\label{hyp-bound}\end{equation}
  hyperbolic equilibria  are $\Gaxi_\mu$-stable for oblate bodies  ($I_1>I_2$)  and unstable for prolate bodies.
\end{prop}
\begin{proof}
Using Proposition \ref{reducedEMsublk} we only need to check for the signs of the eigenvalues \eqref{smale-hyp}, since it is clear that the eigenvalues of the Arnold form \eqref{arnold-hyp} are always positive if $\theta\neq \pm\frac{\pi}{2}$ and  $R>0$. The  eigenvalue $S_2$ in \eqref{smale-hyp} is positive if $R$ satisfies the  large orbit condition \eqref{hyp-bound}. Note that $S_1$ is the only eigenvalue that can change sign. In particular  it has the same sign as $(I_1-I_2)$ and therefore if $I_1>I_2$ the equilibrium is $\Gaxi_\mu$-stable and if $I_1<I_2$ but $R$ satisfies the large orbit condition the equilibrium is unstable using Proposition \ref{negativeeigenvalue}.
\end{proof}

\begin{rem}
Using the simple bound $0<I_2<\frac{1}{2}$ we can check that the right-hand side of \eqref{hyp-bound} is bounded above by $2$. That is, $R$ satisfies the hyperbolic large orbit condition \eqref{hyp-bound} if $R>\sqrt{2}$.
\end{rem}

\subsection{Conical equilibria. Stability}
Following as in previous sections we now apply the Reduced Energy Momentum to the conical family. This results in very long computations but doable using a computer algebra software. Omitting here these details, the series expansion of the Arnold form is diagonal with eigenvalues
\[
A_1 =\frac{R}{\sin^2\psi I_2}+O(R^{-1}), \quad
A_2 =(I_2-I_1)\frac{3\sin^2\psi}{R^3\cos^2\psi}+O(R^{-3})
\]
As for the Smale form we obtain a $2\times 2$ symmetric matrix with entries
\begin{align*}
S_1&=\frac{1}{R^3\sin^2\psi} +\frac{3(I_1-I_2)(3\cos^2\psi-1)}{2R^5 \sin^2\psi}+O(R^{-7})\\
S_2&= \frac{\cos^2\psi}{R\sin^2\psi}+\frac{(45I_1-39I_2)\sin^2\psi+6I_1-2I_2}{2R^3\sin^2\psi}+O(R^{-5})\\
S_{12}&=\frac{\cos\psi}{R^2\sin^2\psi}-\frac{\cos\psi\left((21I_1-23I_2)\sin^2\psi-6I_1+4I_2\right)}{2R^4\sin^2\psi}+O(R^{-6})
\end{align*}
being $S_{12}$ the off-diagonal entry. To check for  positive definiteness of this block we  use Sylvester's criteria. Using again a series expansion we obtain
$$S_1S_2-S_{12}^2=\frac{4I_2-3I_1}{R^6}+O(R^{-8})$$
and now using again Propositions \ref{reducedEMsublk} and \ref{negativeeigenvalue} we obtain the following result.

\begin{prop}
Consider a conical equilibrium with angle $\psi\neq 0,\pm \frac{\pi}{2},\pi$.  If $R$ is large enough then
\begin{itemize}
 \item if $I_1 < I_2$ (prolate body) the conical equilibrium is $\Gaxi_\mu$-stable.
 \item if $I_1 >I_2$ (oblate body) and $4I_2>3I_1$ the conical equilibrium is unstable.
\end{itemize}
In the remaining cases the Reduced Energy Momentum method is inconclusive.
\label{conicalstab}
\end{prop}

\subsection{Isolated equilibria. Stability}
Using the characterization \eqref{isolatedpoint} of the family of isolated equilibria we get $$\lmet \cdot, \cdot \rmet_{(\mathrm{Id},\vecR)}=\begin{bmatrix} I_1 & 0 & 0 & 0 &0 &0\\
0 & I_2+R^2 & 0 & 0 & 0 & -R \\
0 & 0 & I_2+R^2 & 0& R & 0 \\
0& 0 & 0 & 1 & 0 &0 \\
0 &0 & R & 0 & 1& 0 \\
0 & -R & 0 & 0 & 0 & 1\\
  \end{bmatrix},\quad  \II(\text{Id},\vecR)=\begin{bmatrix}
  I_1 & 0 & 0 & -I_1 \\
  0 & I_2+R^2 & 0 & 0 \\
  0 & 0 & I_2+R^2 & 0 \\
  -I_1 & 0 & 0 & I_1 \\
  \end{bmatrix}.
  $$
Therefore the angular momentum and its stabilizer are
\[\mu=\II(\text{Id},\vecR)(\xi,\eta)=(I_2+R^2)\xi_3 \e_3\in(\gaxi)^* \implies \gaxi_\mu = \Big\langle \e_3,\ \e_4 \Big\rangle. \]
Note that we also have
\[\gaxi_{q}=\Big\langle\e_1+\e_4 \Big\rangle \subset \gaxi\]
and this is precisely the kernel of the locked inertia tensor.

As the action of $\Gaxi$ at the equilibrium point $q=(\mathrm{Id},\vecR)\in Q$ is not locally free
we cannot apply the Energy Momentum method as described in Proposition \ref{reducedEMsublk}. To study its stability we will need the singular version stated in Proposition \ref{singularVamended}. As we have already computed $\gaxi_{q}$ and $\gaxi_\mu$ we can check that
\[\mathfrak{t}=\Big\langle \e_2\Big\rangle,\quad \mathfrak{q}^\mu=0\quad\text{and}\]
\[\Sigma=\Big\langle\e_4,-R\e_3+(I_2+R^2)\e_5,R\e_2+(I_2+R^2)\e_6\Big\rangle\subset T_{(\mathrm{Id},\vecR)}Q\]
satisfy the conditions of Proposition \ref{singularVamended}. After some substitutions we find that $(\ed^2 V_{(\xi,0)}+\mathrm{corr}_{(\xi,0)})\restr{\Sigma}$ at the isolated equilibrium is diagonal with respect to the chosen basis and has eigenvalues
\begin{align*}
H_1&=(R^4+3R^2(I_1-2I_2)+15I_2(I_1-I_2))\frac{1}{R^5(R^2+I_2)} \\
H_2&=3(I_2-I_1)(R^2+I_2)^2R^{-5} \\
H_3&=\big( R^2(4I_2-3I_1)-6I_2(I_1-I_2) \big)\frac{R^2+I_2}{R^5}.
\end{align*}
From these computations we can obtain the following stability result.
\begin{prop}
Consider an isolated equilibrium with a large orbit satisfying
$$R^4>3R^2(2I_2-2I_1)+15I_2(I_2-I_1).$$
Then
\begin{itemize}
 \item if $I_1<I_2$ (prolate bodies) the equilibrium is $G_\mu$-stable.
 \item if $I_1>I_2$ and $4I_2-3I_1>\frac{6I_2(I_1-I_2)}{R^2}$ the equilibrium is unstable.
\end{itemize}
For the remaining cases the method is inconclusive. \label{isolatedstab}
\end{prop}

\subsection{Isotropy subgroups}
\label{axi-isotropy}

As we did in Section \ref{asym-isotropy} we will use the different isotropy groups  to classify the different families of relative equilibria for the axisymmetric problem. Additionally, this can be used to discard the existence of some bifurcations based on symmetry considerations. This classification is shown in the following table

\vspace{4mm}
\renewcommand{\arraystretch}{1.5}
\begin{tabular}{| c| l | l |}
\hline
 \multirow{4}{2cm}{\centering cylindrical $\alpha\neq 0$} &  $\vecR=(0,R,0)$         &$\Gaxi_q=\{(A\det(A),A)\in \Gaxi \mid  A=\diag(\pm 1,1,\pm 1)\}\cong \Z_2\times \Z_2$ \\
             &  $\xi=(|\xi|,0,0)$         &$\Gaxi_\mu=\{(\exp(t\e_1),A)\in \Gaxi \mid t\in \R,\ \chi(A)=1\}\cong S^1\times S^1\times \Z_2$ \\
             & $\eta=-\alpha|\xi|$        &$\Gaxi_z=\{(A\det(A),A)\in \Gaxi \mid A=\diag(\pm 1,1,1)\} \cong \Z_2$ \\
             & $\mu=(\mu_1,0,0;\mu_4)$ & \\
\hline
\multirow{4}{2cm}{\centering cylindrical $\alpha= 0$}& $\vecR=(0,R,0)$ & $\Gaxi_q=\{(A\det(A),A)\in \Gaxi \mid  A=\diag(\pm 1,1,\pm 1)\}\cong \Z_2\times \Z_2$ \\
             & $\xi=(|\xi|,0,0)$ &$\Gaxi_\mu=\{(\exp(t\e_1),A)\in \Gaxi \mid t\in \R\}\cong S^1\times O(2)\times \Z_2$ \\
             & $\eta=0$ &$\Gaxi_z=\{(A\det(A),A)\in \Gaxi \mid A=\diag(\pm 1,1,1)\} \cong \Z_2$ \\
             & $\mu=(\mu_1,0,0;0)$    & \\
\hline
\multirow{4}{2cm}{\centering hyperbolic $\theta\neq 0,\pi$}   & $\vecR=(0,R,0)$ &$\Gaxi_q=\{(A\det(A),A)\in \Gaxi \mid  A=\diag(\pm 1,1,\pm 1)\}\cong \Z_2\times \Z_2$ \\
             & $\xi=(\xi_1,0,\xi_3)$ &$\Gaxi_\mu=\{(\exp(t\mu_{1,2,3}),A)\in \Gaxi \mid t\in \R,\ \chi(A)=1\}\cong S^1\times S^1\times \Z_2$ \\
             & $\eta\neq 0$ &$\Gaxi_z=\{(\mathrm{Id},\mathrm{Id})\}$ \\
             & $\mu=(\mu_1,0,\mu_3;\mu_4)$   & \\
\hline
\multirow{4}{2cm}{\centering hyperbolic $\theta=0$}  & $\vecR=(0,R,0)$ &$\Gaxi_q=\{(A\det(A),A)\in \Gaxi \mid  A=\diag(\pm 1,1,\pm 1)\}\cong \Z_2\times \Z_2$ \\
             &  $\xi=(0,0,|\xi|)$ &$\Gaxi_\mu=\{(\exp(t\e_3),A)\in \Gaxi \mid t\in \R\}\cong S^1\times O(2)\times \Z_2$ \\
             &  $\eta=0$ &$\Gaxi_z=\{(A\det(A),A)\in \Gaxi \mid A=\diag(1,1,\pm 1)\} \cong \Z_2$ \\
             &  $\mu=(0,0,\mu_3;0)$   & \\
\hline
\multirow{4}{*}{conical}      & $\vecR=(R_1,R_2,0)$ &$\Gaxi_q=\{(A\det(A),A)\in \Gaxi \mid  A=\diag(1,1,\pm 1)\}\cong \Z_2$ \\
             & $\xi=(\xi_1,\xi_2,0)$ &$\Gaxi_\mu=\{(\exp(t\mu_{1,2,3}),A)\in \Gaxi \mid t\in \R,\ \chi(A)=1\}\cong S^1\times S^1\times \Z_2$ \\
             & $\eta\neq 0$ &$\Gaxi_z=\{(\mathrm{Id},\mathrm{Id})\}$ \\
             & $\mu=(\mu_1,\mu_2,0;\mu_4)$   & \\
\hline
\multirow{4}{*}{isolated}     & $\vecR=(R,0,0)$ &$\Gaxi_q=\{(A\det(A),A)\in \Gaxi \mid  A\e_1=\e_1\}\cong O(2)$ \\
             & $\xi=(0,|\xi|,0)$ &$\Gaxi_\mu=\{(\exp(t\e_2),A)\in \Gaxi \mid t\in \R\}\cong S^1\times O(2)\times \Z_2$ \\
             & $\eta=0$ &$\Gaxi_z=\{(A\det(A),A)\in \Gaxi \mid A=\diag(1,\pm 1, 1)\} \cong \Z_2$ \\
             & $\mu=(0,\mu_2,0;0)$   & \\
\hline
\end{tabular}
\renewcommand{\arraystretch}{1}

Regarding $\Ar$ and $\Sm$, it is the case that  the diagonal structures found in the cylindrical, hyperbolic and conical case are again due to the fact that the different bases chosen had good isotypic properties. In fact, in the conical equilibrium we cannot, based on symmetry arguments,  choose a  basis in which $\Sm$ is  diagonal because $G_z^\text{axi}$ acts trivially on $\Vrig$ and therefore there is only one isotypic block which is the whole space.

\subsection{Bifurcations}
\label{sec:axi-bifurcations}

Before we study the bifurcations of the axisymmetric equilibria we must remark that the parametrization of hyperbolic relative equilibria \eqref{hyperbolicpoint} in terms of $R$ and $\theta$ is not one-to-one in the sense defined in the Appendix, since if we consider the point
\begin{equation*}\xi=|\xi|\begin{bmatrix}\sin\theta \\0 \\ \cos\theta \end{bmatrix},\quad \mathbf{R}=\begin{bmatrix}0 \\R \\ 0 \end{bmatrix},\quad\eta=-\frac{I_2-I_1}{I_1}|\xi|\sin\theta \end{equation*}
and we translate it using $(\diag(-1,-1,1),\diag(1,1,-1))\in \Gaxi$ we get
\begin{equation*}\xi=|\xi|\begin{bmatrix}-\sin\theta \\0 \\ \cos\theta \end{bmatrix},\quad \mathbf{R}=\begin{bmatrix}0 \\R \\ 0 \end{bmatrix},\quad\eta=\frac{I_2-I_1}{I_1}|\xi|\sin\theta \end{equation*}
therefore $(R,\theta)$ and $(R,-\theta)$ parametrize the same relative equilibrium up to a $\Gaxi$-symmetry. Similarly the transformation $(\diag(1,-1,-1),\diag(-1,1,1))\in \Gaxi$ is equivalent to the map $(R,\theta)\mapsto (R,\pi-\theta)$.
It can be checked that if we restrict $\theta\in[0,\frac{\pi}{2})$ then  relative equilibria with different $\theta$ are not $\Gaxi$-related.
The same problem happens along the conical family: In order to have injectivity we must restrict $\psi\in (0,\frac{\pi}{2})$. Nevertheless as we will see it is clearer to study the bifurcation problem ignoring this injectivity issue and considering the set of hyperbolic equilibria as a set parametrized by $\theta\in S^1$ and the orbital radius $R$, and the set of conical equilibria as a set parametrized by $\psi \in S^1$ with $\psi\neq 0,\pm \frac{\pi}{2},\pi$ and the orbital radius $R$.

Although  the defining equations  \eqref{eq:eta} and \eqref{eq:lambda} for a conical equilibrium with a given $\psi$ are not defined for $\psi=0$ we can consider the limit of the family of conical equilibria with fixed $R$ as $\psi\to 0$. This limit  converges to the point $(\text{Id},\vecR, (\xi, \eta))$ with
\begin{equation}\vecR = (R,0,0),\quad \xi=(0,|\xi|,0),\quad \eta=0 \label{conic0}\end{equation}
where $|\xi|^2=\frac{2R^3+3-9I_1}{2R^5}$. Note that this point  lies in what we have called the isolated family \eqref{isolatedpoint}.
Analogously the limiting point  $\psi\to\pi$ corresponds to $\vecR = (R,0,0),\quad \xi=(0,-|\xi|,0),\quad \eta=0$ a relative equilibrium related to \eqref{conic0}
by the symmetry $(\diag(1,-1,-1),\diag(1,-1,-1))\in \Gaxi$.

In a similar way it can be shown that the limit when $\psi\to\frac{\pi}{2}$ is
\begin{equation}\vecR = (0,R,0),\quad \xi=(|\xi|,0,0),\quad \eta=-\alpha_{\text{conic}}|\xi| \label{conicpi2}\end{equation}
where $\alpha_{\text{conic}}=\frac{-(I_1-I_2)(8R^2+9I_1-9I_2)}{I_1(2R^2+9I_1-9I_2)}$. Note that this point is a cylindrical equilibrium with spinning rate $\alpha=\alpha_{\text{conic}}$. The point  $\psi\to-\frac{\pi}{2}$ is related to $\psi\to\frac{\pi}{2}$ using the element  $(\diag(-1,-1,1),\diag(1,1,-1))\in \Gaxi$.

To start the bifurcation analysis for axisymmetric relative equilibria we will study what happens to the cylindrical family when $R$ is large enough and we vary the spinning quotient $\alpha$ (see \eqref{cylindricalpoint}). As we previously, we need to find the points for which the Arnold  or the Smale forms become degenerate, since these are necessary conditions for the existence of bifurcations. We will study separately what happens in the oblate and prolate cases.

\paragraph{Oblate.} In this case we will fix the orbital radius $R=|\vecR|$ to a large enough value and  will study the existence of the  different kinds of relative equilibria. The different transitions have been sketched in Figure \ref{oblatetransf} where a thick solid line means $\Gaxi_\mu$-stability, a dotted line means instability and a dashed line means that the methods employed give inconclusive results. In that figure each line represents a connected component of each family of relative equilibria in $T^*Q$ with the orbital radius $R=|\vecR|$ fixed to a large enough value.

Since in the oblate case $I_1>I_2$ then for any positive $\alpha$ the cylindrical equilibrium is $\Gaxi_\mu$-stable according to Proposition \ref{cylstab}. Note that if $\alpha$ is decreased we eventually arrive to the point $P_9$ where $\alpha=0$ and according to \eqref{cylindricalpoint} the spinning velocity $\eta$ vanishes but the cylindrical equilibrium is still $\Gaxi_\mu$-stable.
If $\alpha$ is  further decreased we will reach the point $P_1$ where $I_1(1+\alpha)-I_2=0$. At this point $A_1$ (see \eqref{arnold-cyl}) becomes negative and therefore the cylindrical family becomes unstable because there is only one negative eigenvalue of $\Ar$ and $\Sm$ (Proposition \ref{negativeeigenvalue}).

The cylindrical equilibrium will remain unstable until we arrive to $P_5$ where $S_1=0$ ($I_1(4+\alpha)-4I_2\approx 0$).
 After this point two eigenvalues are negative and therefore we cannot say anything about its stability without recurring to a numerical analysis of the linearized system, and so  the inconclusive region starts at that point.

Note that according to Remark \ref{rem:int-cyl-hyp} the point $P_1$, for which $I_1(1+\alpha)-I_2=0$, is a relative equilibrium belonging  to the cylindrical family, and corresponding to a limit point of the hyperbolic family when $\theta\to\frac{\pi}{2}$.

Recall that using Proposition \ref{hypstab}
 the hyperbolic family is stable. Therefore we can interpret this behavior saying that at $P_1$ the $\Gaxi_\mu$-stability has been transferred from the cylindrical family to the hyperbolic family.

The other bifurcation candidate along the cylindrical family is $P_5$, where $S_1$ changes sign. This happens exactly when \[\alpha=\frac{-(I_1-I_2)(8R^2+9I_1-9I_2)}{I_1(2R^2+9I_1-9I_2)}\approx 4\frac{I_2-I_1}{I_1},\] but this is limit point of the conical family when $\psi \to \pm \frac{\pi}{2}$ (see \eqref{conicpi2}).

Now we  consider what happens if we move along the conical family changing the value of $\psi$: For $\psi \to \frac{\pi}{2}$ the limiting equilibrium point belongs to the cylindrical family. As $\psi$ becomes smaller, the body starts to orbit in a higher plane, and the plane of the orbit now does not contain the center of attraction (see \eqref{offset}). At the same time  the spinning velocity $\eta$ increases.

 The offset $\varkappa$ (see \eqref{defoffset}) achieves a maximum for $\psi=\frac{\pi}{4}$, and passed that value  the orbit plane starts to descend until $\psi\to 0$ at the point $P_6$ which belongs to the isolated family  \eqref{isolatedpoint} and for which the offset $\varkappa=0$ vanishes again. If $\psi$ is  further decreased, the orbit plane lowers down and the body describes again a cone.
For $\psi\to-\frac{\pi}{2}$  (point $P_7$) the we have again a cylindrical equilibrium,  and this point is the image of $P_5$ by translating with a $\Gaxi$-symmetry.

In general the stability along the conical and isolated equilibria will depend on the sign of $4I_2-3I_1$.
 If this quantity is positive in both cases we will have only one negative eigenvalue and using Proposition \ref{negativeeigenvalue} we can conclude instability. But if $4I_2-3I_1<0$
 we will have two negative eigenvalues in both cases and the method will be inconclusive.

Finally, if we move along the hyperbolic equilibria, it follows from Proposition \ref{hypstab}  that the family is $\Gaxi_\mu$-stable and the only bifurcation candidates are $\theta\to\pm\frac{\pi}{2}$ where the bifurcation to hyperbolic and cylindrical families occur. Note that according to \eqref{hyperbolicpoint} when $\theta=0$ the spinning velocity $\eta$ vanishes, corresponding to the point $P_2$.

The points $P_3,P_4,P_7,P_8$ and $P_{10}$ are $\Gaxi$-related  to $P_3,P_2,P_5,P_6$ and $P_9$ respectively, and therefore they have the same dynamical properties as their $\Gaxi$-counterparts.

\begin{figure}[h]
\begin{center}
\vspace{10mm}
\includegraphics[width=10cm]{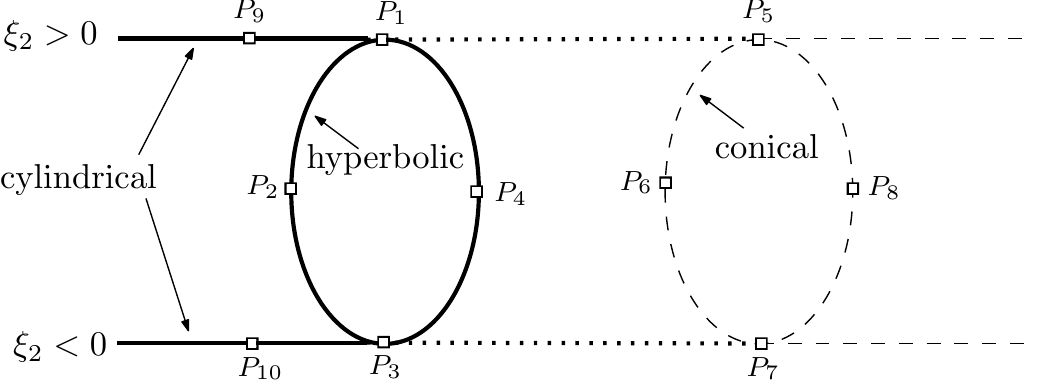}
\caption{Bifurcation diagram for an oblate body.\label{oblatetransf}}
\end{center}
\end{figure}

\paragraph{Prolate.} If the body is prolate ($I_1<I_2$) the bifurcation diagram is very similar as the previous one.  The hyperbolic and the conical family bifurcate from the cylindrical family as $\alpha$ is varied,
but the difference is that now the bifurcation point with largest $\alpha$ is $P_5$ and not  $P_1$. As both bifurcations happen at positive $\alpha$
 the non-spinning point $P_9$ lies after the hyperbolic bifurcation.

In this prolate case,  $\Gaxi_\mu$-stability is transferred from the cylindrical  to the conical family. The hyperbolic equilibria are now unstable
and the bifurcation diagram is shown in Figure \ref{prolatetransf} where we have used  the same notation as for the oblate case.

\begin{figure}[h]
\begin{center}
\vspace{10mm}
\includegraphics[width=10cm]{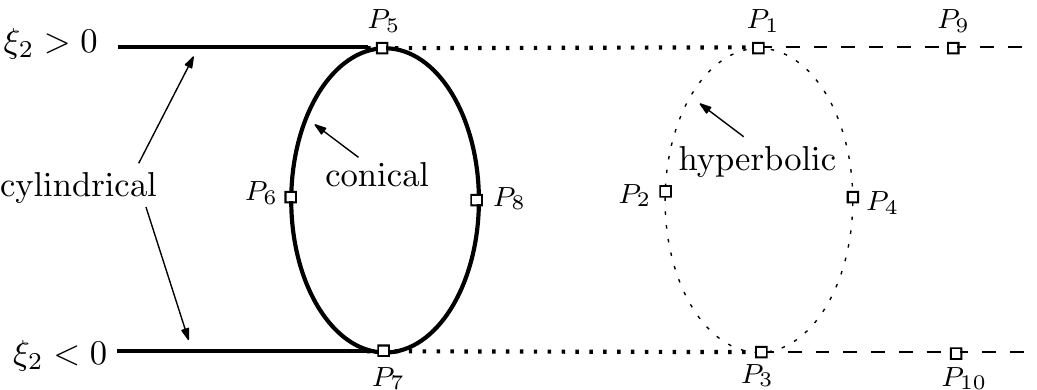}
\caption{Bifurcation diagram for a prolate body.\label{prolatetransf}}
\end{center}
\end{figure}

\appendix
\addappheadtotoc

\section{Hamiltonian Relative Equilibria}

Let $(Q,\lmet\cdot,\cdot\rmet)$ be a Riemannian manifold (the
configuration manifold), $G$ a compact Lie group that acts by isometries on
$Q$ (the symmetry group) and $V\in C^\infty (Q)$ a $G$-invariant
function (the potential energy). With these data we can construct
a symmetric Hamiltonian system on $T^*Q$, equipped with its canonical symplectic form $\omega=-\ed \theta$, in the following way: The
potential energy $V$ can be lifted to $T^*Q$,  and we will denote
this lifted function also by $V$. The Riemannian metric on $Q$
induces an inner product on each cotangent fiber $T^*_qQ$, $q\in Q$.
Then the Hamiltonian is defined as
\begin{equation*}h(p_q)=\frac 12 \|p_q\|^2+V(q),\quad p_q\in T_q^*M.\end{equation*} The
$G$-action on $Q$ induces a cotangent-lifted Hamiltonian action on
$T^*Q$ with associated equivariant momentum map $\J:T^*M\rightarrow
\g^*$ defined by
\begin{equation*}\langle \J(p_q),\xi\rangle=\langle
p_q,\xi_Q(q)\rangle\quad\forall\,\xi\in\g\end{equation*}
where $\xi_Q$
is the fundamental vector field on $Q$ associated to the generator
$\xi\in\g$, and is defined by
$$\xi_Q(q)=\frac{d}{dt}\restr{t=0}e^{t\xi}\cdot q.$$
Analogously $\xi_{T^*Q}$ is the fundamental vector field of the cotangent lifted action on $T^*Q$.
This momentum map is $\text{Ad}^*$-equivariant in the sense that
$\J(g\cdot p_q)=\Ad_{g^{-1}}^*(\J(p_q))$ for every $p_q\in T_q^*M$,
$g\in G$.

 The Hamiltonian $h$ is $G$-invariant for this lifted action (this
follows from the invariance of the metric and of $V$). Therefore,
due to Noether's theorem, $\J$ is conserved for the Hamiltonian dynamics associated to $h$. The
quadruple $(Q,\lmet\cdot,\cdot\rmet,G,V)$ is called a \emph{symmetric
simple mechanical system}.

Let $(Q,\lmet\cdot,\cdot\rmet,G,V)$ be a symmetric simple mechanical system. A \emph{relative equilibrium} is a point in phase space
$p_q\in T^*Q$ such that its dynamical evolution lies inside a group
orbit for the cotangent-lifted action. This amounts to the existence
of a generator $\xi\in\g$ such that the evolution of $p_q$ is given
by $t\mapsto e^{t\xi}\cdot p_q$. The element $\xi$ is called the
\emph{angular velocity} of the relative equilibrium.
A useful characterization of relative equilibria in simple
mechanical systems is given by the following result.
\begin{thm}[Marsden \cite{marsden1992lectures}]  \label{REthm}A point  $p_q\in T^*Q$ of a symmetric simple mechanical system $(Q,\lmet\cdot,\cdot\rmet,G,V)$ is a relative equilibrium  with  velocity $\xi\in\g$
if and only if the following conditions are verified
\begin{enumerate}
\item $p_q = \lmet \xi_Q(q) , \cdot\rmet$.
\item  $q$ is a critical point of $V_\xi(q)=V(q)-\frac 12\langle\xi,\II(q)\xi\rangle$,
\end{enumerate}
where $\II:Q\times\g\rightarrow\g^*$ is defined by
$\langle\xi,\II(q)\eta\rangle=\lmet\xi_Q(q),\eta_Q(q)\rmet$.
Moreover, the momentum $\mu=\J(p_q)\in\g^*$ of the relative
equilibrium is given by $\mu = \II(q)
\xi. $
\end{thm}
That is, any relative equilibrium
is determined by a configuration-velocity pair $(q,\xi)\in
Q\times\g$ satisfying $\ed V_\xi(q)=0$. The map $\II$ is called the
\emph{locked inertia tensor}, while the function $V_\xi$ is called
the \emph{augmented potential}.

Note that if $(q,\xi)$ is a relative equilibrium pair, by $G$-invariance of the Hamiltonian vector field $t\mapsto g\cdot (e^{t\xi}\cdot p_q)=e^{t\Ad _g\xi}\cdot(g\cdot p_q)$ is also an integral curve, that is, for each $g\in G$ if $(q,\xi)$ is a relative equilibrium pair then $(g\cdot q,\Ad_g\xi)$ is also a relative equilibrium pair.

\paragraph{The Reduced Energy Momentum Method.}\label{sec reducedEM}

The  generally adopted notion of stability for relative equilibria
of symmetric Hamiltonian systems  is that of $G_\mu$-stability,
introduced in \cite{patrick1992stability} and that we now review in the
context of symmetric simple mechanical systems. This notion  is
closely related to the one of Liapunov stability of the induced Hamiltonian
system on the reduced phase space.
\begin{defn}\label{patrick} Let $(Q,\lmet\cdot,\cdot\rmet,G,V)$ be a symmetric simple mechanical system and $p_q\in T^*Q$ a
relative equilibrium  with momentum value $\mu=\J(p_q)$. Let $G_\mu$ be the stabilizer of $\mu\in \g^*$ for the coadjoint action of $G$ in $\g^*$. We say that
$p_q$ is $G_\mu$-stable if for every $G_\mu$-invariant neighborhood
$U\subset T^*Q$ of the orbit $G_\mu\cdot p_q$ there exists a
neighborhood $O$ of $p_q$ such that the Hamiltonian evolution of $O$
lies in $U$ for all time.
\end{defn}

In this section we outline the implementation of the Reduced Energy Momentum method following \cite{simo1991stability}. This method is designed for simple mechanical systems since it incorporates all
of their distinguishing  characteristics with respect to general
Hamiltonian systems. Relative equilibria pairs $(q,\xi)$, besides being critical points of the augmented potential can also be described as relative equilibria position-momentum pairs $(q,\mu)\in Q\times \g^*$ where $\mu=\II(q)\xi$. This pairs $(q,\mu)$ are precisely the critical points of the \emph{amended potential}

$$V_\mu(q)=V(q)+\frac{1}{2}\langle\mu,\II(q)^{-1}\mu\rangle$$

The study of the augmented potential $V_\xi$ can only give rough stability conditions, and a better function to study in order to determine the stability of a relative equilibrium  is the amended potential $V_\mu$. From now on we will need to assume that the isotropy, or stabilizer,  $G_q$ of the base point of the relative equilibrium is discrete.

\begin{prop}(\cite{patrick1992stability})
Given a relative equilibrium $p_q$ with momentum $\mu$ and $G_q$ discrete, definiteness of the bilinear form $\ed ^2 V_\mu(q)\rrestr{\mathcal{V}}$ where $\mathcal V=\{v\in T_qQ\ |\ \lmet v,\eta_Q\rmet=0 \quad \forall \eta \in \g_\mu \}$ implies $G_\mu$-stability. \label{amendedtest}
\end{prop}

An additional benefit of the Reduced Energy Momentum method is to further exploit the symmetry properties of $V_\mu$ and decompose the admissible variation space $\mathcal V$ into two subspaces \begin{equation}\mathcal V= \Vrig \oplus\Vint\label{splittingREM}\end{equation} in such a way that $\ed^2V_\mu$ block-diagonalizes.

Let $\g_\mu^\perp\subset \g$ be the orthogonal complement of $\g_\mu$ with respect to $\II(q)$. We define
\begin{align}
\Vrig&=\{\eta_Q(q)\in T_{q}Q \mid \eta \in \g_\mu^\perp\}\subset \mathcal V \nonumber\\
\Vint&=\{v \in \mathcal V \mid \eta \cdot \langle\ed (\mathbb I\xi),v\rangle=0 \quad \forall\eta \in \g_\mu^\perp\}=\{v \in \mathcal V \mid \mathbb I(q)^{-1} \langle\ed (\mathbb I\xi),v\rangle \in \g_\mu\} \label{Vint}
\end{align}
(an interpretation of these spaces can be found in \cite{marsden1992lectures}). The \emph{Arnold form} $\Ar:\g_\mu^\perp \times \g_\mu^\perp\to \R$ is the bilinear form defined as the restriction of the second variation $\ed^2 V_\mu$ to $\Vrig\cong \g_\mu^\perp$. It can be shown that the Arnold form is independent of the potential and can be written as
\begin{equation}\label{arnold}\Ar(\eta,\nu):=\ed^2V_\mu(q)(\eta_Q(q),\nu_Q(q))=\langle\text{ad}^*_\eta \mu, \II(q)^{-1}\text{ad}^*_\nu \mu\rangle+\langle\text{ad}^*_\eta\mu,\text{ad}_\nu(\II(q)^{-1}\mu)\rangle.\end{equation}

If the Arnold form is non-degenerate it can be shown that $\mathcal V=\Vrig\oplus \Vint$. In that case sufficient conditions for stability are given by the following result.

\begin{prop}(Reduced Energy Momentum method \cite{simo1991stability}) Given a relative equilibrium $p_q$ with $G_q$ discrete then
\begin{itemize}
 \item If $\dim G < \dim Q$ then positive definiteness of both $\Sm:=\ed^2 V_\mu \restr{\Vint}$ and  $\Ar$  implies $G_\mu$-stability of the relative equilibrium.
 \item If $\dim G=\dim Q$ then definiteness (positive or negative) of  $\Ar$  implies $G_\mu$-stability of the relative equilibrium.
\end{itemize}

\label{reducedEMsublk}
\end{prop}
The bilinear form $\Sm:\Vint\times \Vint \to \R$ is known in the literature as the \emph{Smale form}.
The Reduced Energy Momentum method can also be used to study the linearized dynamics near a relative equilibrium therefore offering also instability conditions. Our main instability result will be the following.

\begin{prop} Consider a relative equilibrium $p_q$. If the total number of negative eigenvalues of $\Sm$ and $\Ar$ is odd then the relative equilibrium is \emph{unstable}.
\label{negativeeigenvalue}
\end{prop}
\begin{proof}
 We will combine symplectic eigenvalue analysis and the Reduced Energy Momentum splitting. The linearized vector field with respect to a symplectic slice (see \cite{patrick1995relative}) is an infinitesimally symplectic transformation with eigenvalues coming in $n$-tuples. This implies that if $\lambda=a+bi$ is a simple eigenvalue of this linearized field then $\{\bar{\lambda},-\lambda,-\bar{\lambda}\}$ are the companion eigenvalues completing a quadruple. If $a\neq 0$ and $b\neq 0$ all the four eigenvalues are different and their product is $|\lambda|^4$. If $a=0$ then the eigenvalue pair is $\{\lambda,\bar{\lambda}\}$  and their product is $|\lambda|^2$ but if $b=0$ then the  eigenvalue pair is $\{\lambda,-\lambda\}$ and the product is $-|\lambda|^2$. Therefore if the linearized vector field has negative determinant then it must have a real eigenvalue implying instability. Additionally, it can be shown (see \cite{marsden1992lectures}) that the  splitting \eqref{splittingREM} block-diagonalizes the
symplectic form. Moreover the sign of
the determinant of the linearized field has to be equal to the sign of the determinant of $\ed^2 V_\mu\restr{\mathcal{V}}$. Thus if $\ed^2 V_\mu\restr{\mathcal{V}}$ has negative determinant then the relative equilibrium is unstable. But the Reduced Energy Momentum method block diagonalizes $\ed^2 V_\mu\restr{\mathcal{V}}$ into $\Sm$ and $\Ar$ and the hypothesis implies that the determinant of $\ed^2 V_\mu\restr{\mathcal{V}}$ being negative is equivalent to the total number of negative eigenvalues of $\Sm$ and $\Ar$ being odd.
\end{proof}

\paragraph{Bifurcations of relative equilibria.}

By a family of relative equilibria of a  symmetric simple mechanical system we will mean the image of a map $\psi:\R^k\to T^*Q$ such that for each $x\in \R^k$ the point $\psi(x)$ is a relative equilibria and if $x\neq y$ then $\psi(x)$ and $\psi(y)$ are not $G$-related.

Combining the Reduced Energy Momentum and the results of \cite{patrick1995relative} it can be shown that if $p_q\in T^*Q$ is a relative equilibrium with discrete stabilizer such that both the Arnold form and the Smale form are non-degenerate then there is a family of dimension $\dim G_\mu$ of relative equilibria containing $p_q$ and all the relative equilibria near $p_q$ lie in that family.

Two families $\psi_1,\psi_2$ are different if there are not $x,y$ such that $\psi_1(x)$ and $\psi_2(y)$ are $G$-related. The previous result implies that at a bifurcation point  the Arnold form or the Smale form must be degenerate. Therefore the points which are candidates for the existence of a bifurcation  are those points at which at least one of these forms becomes degenerate. However, in order to decide if actually a bifurcation happens at one of these points a local study around each candidate by other methods is needed.

\paragraph{Singular Energy Momentum Method (case of $Q$-isotropy only).}

As the action of $\Gaxi$ on the isolated equilibria points of the axisymmetric problem is not locally discrete we can not apply the Energy Momentum method described before. We will use the singular Reduced Energy Momentum method developed in \cite{rodriguez2006stability} to cover that case.
However for our purposes we will only need a particular version of the method that we will briefly outline. Let $(Q,\lmet,\rmet,G,V)$ be a symmetric simple mechanical system and $p_q\in T^*Q$ a relative equilibrium with $\g_{p_q}=\{0\}$, and $\g_q\neq \{0\}$. Define the subspaces
\begin{itemize}
 \item $\mathfrak{t}\subset \g$ such that $\g=\g_q\oplus \g_\mu\oplus \mathfrak{t}$ and $\g_\mu\perp\mathfrak{t}$ with respect to $\II(q)$.
 \item $\mathbf{S}=(\g\cdot q)^\perp \subset T_qQ$ with respect to $\lmet\cdot,\cdot\rmet_q$.
 \item $\mathfrak{q}^\mu:=\{\gamma\in \mathfrak{t}\mid \langle\ad^*_\gamma\mu,\xi\rangle=0,\quad \forall \xi\in \g_q\}$
 \item $\text{corr}_\xi(q)(v_1,v_2)=\big\langle\Proj((\D_{v_1}\II)(\xi)),(\Proj^T\II(q)\Proj)^{-1}\Proj((\D_{v_2}\II)(\xi))\big\rangle$ where $\Proj:\g^*\to \mathfrak{g}_\mu^*\oplus\mathfrak{t}^*$ is the linear projection
 and $\D_v\II:\g\to \g^*$ is the directional derivative of $\II$ with direction $v\in T_qQ$.
 \item $\Sigma:=\{\xi_Q(q)+a\mid \xi\in \mathfrak{q}^\mu,\quad a\in \mathbf{S}\}$.
\end{itemize}
With these definitions the analogue of Proposition \ref{amendedtest} is
\begin{prop}
(Theorem 4.1 of \cite{rodriguez2006stability}) Let $p_q\in T^*Q$ be a relative equilibrium with configuration velocity pair $(q,\xi)\in Q\times \g$. Assume that $\g_{p_q}=\{0\}$ and $\dim(G\cdot q)<\dim (M)$.

If $(\ed^2_q V_\xi+\mathrm{corr}_\xi(q))\restr{\Sigma}$ is positive definite then the relative equilibrium is $G_\mu$-stable. \label{singularVamended}
\end{prop}

\end{document}